\documentclass{amsart}

\usepackage{amsmath,amssymb,amsthm,mathrsfs,amsfonts}

\usepackage{color}

\def\a{\alpha}
\def\b{\beta}
\def\d{\delta}
\def\de{\Delta}

\def\ga{\Gamma}
\def\iy{\infty}

\def\l{\lambda}
\def\la{\Lambda}
\def\s{\sigma}

\def\t{\tau}

\def\o{\omega}
\def\om{\Omega}

\def\ts{\times}

\def\iy{\infty}

\def\diag{{\rm diag\, }}

\def\BD{{\mathbb D}}

\def\BZ{{\mathbb Z}}

\newcommand{\sD}{{\mathcal D}}
\newcommand{\sE}{{\mathcal E}}\newcommand{\sF}{{\mathcal F}}
\newcommand{\sG}{{\mathcal G}}\newcommand{\sH}{{\mathcal H}}

\newcommand{\sK}{{\mathcal K}}\newcommand{\sL}{{\mathcal L}}

\newcommand{\sU}{{\mathcal U}}
\newcommand{\sX}{{\mathcal X}}
\newcommand{\sY}{{\mathcal Y}}

\newcommand{\eC}{{\mathbf C}}

\newcommand{\eM}{{\mathbb M}}
\newcommand{\eUM}{{\mathbb{UM}}}
\newcommand{\asU}{{\mathbf{U}}}
\newcommand{\asX}{{\mathbf{X}}}
\newcommand{\asY}{{\mathbf{Y}}}
\newcommand{\asD}{{\mathbf{D}}}
\newcommand{\asE}{{\mathbf{E}}}
\newcommand{\asF}{{\mathbf{F}}}
\newcommand{\asZ}{{\mathbf{Z}}}

\newcommand{\bpr}{{\noindent\textbf{Proof.}\ \ }}
\newcommand{\epr}{{\hfill $\Box$}}

\newtheorem{thm}{Theorem}[section]

\newtheorem{cor}[thm]{Corollary}
\newtheorem{prop}[thm]{Proposition}
\newtheorem{probl}{Problem}[section]
\newtheorem{proc}{Procedure}[section]

\newcommand{\half}{\frac{1}{2}}
\newcommand{\K}{K}
\newcommand{\mat}[2]{\ensuremath{\left[\begin{array}{#1}
#2
\end{array} \right]}}
\newcommand{\ov}[1]{{\overline{#1}}}

\newcommand{\ands}{\quad\mbox{and}\quad}

\newcommand{\tu}[1]{\textup{#1}}
\newcommand{\re}{\textup{Re\,}}

\date{}

\begin{document}

\title[A time-variant norm constrained interpolation problem]{A time-variant norm constrained
interpolation problem arising from relaxed commutant lifting}

\author[A.E. Frazho]{A.E. Frazho}

\address{%
Department of Aeronautics and Astronautics\\
Purdue University\\
West Lafayette, IN 47907, USA}

\email{frazho@ecn.purdue.edu}


\author[S. ter Horst]{S. ter Horst}

\address{%
Department of Mathematics\\
Virginia Tech\\
Blacksburg VA 24061, USA}

\email{terhorst@vt.edu}


\author[M.A. Kaashoek]{M.A. Kaashoek}

\address{%
Afdeling Wiskunde,
Faculteit der Exacte Wetenschappen\\
Vrije Universiteit\\
De Boelelaan 1081a, 1081 HV Amsterdam, The Netherlands}

\email{ma.kaashoek@few.vu.nl}




\begin{abstract}
A time-variant analogue of an interpolation problem equivalent to
the relaxed commutant lifting problem is introduced and studied. In
a somewhat less general form the problem already appears in the
analysis of the set of all solutions to the three chain completion
problem. The interpolants are upper triangular operator matrices of
which the columns induce contractive operators.  The set of all
solutions of the problem is described explicitly. The results
presented are time-variant analogues of the main theorems in
\cite{FtHK07b}.
\end{abstract}

\maketitle

\vspace{-1cm}
\setcounter{section}{-1}\setcounter{equation}{0}
\section{Introduction}\label{sec:intro}

Time-variant versions of metric constrained interpolation problems
and time-\linebreak varying linear system theory have been
intensively studied since the early 1990's; see the papers
\cite{ADw90,ADwD90,BGK92,BGK94,Dw08} and the books \cite{DwKV93,
HI94, FFGK98,DwvdV98} for a general overview and additional
references. The connection with commutant lifting theory was made in
\cite{FFGK97a}, where a time-varying analogue of the commutant
lifting theorem, known as the three chain completion theorem, was
proved. An early version of a time-variant commutant lifting theorem
appeared in Ball-Gohberg \cite{BG85}, which was later extended to
the setting of nest algebras in \cite{PP88} (see also \cite{D88});
the connection with the three chain theorem is explained in
\cite{B96}. One of the recent developments in commutant lifting
theory is the introduction of a relaxation of the commutant lifting
setting in \cite{FFK02}. In the present paper we consider a
time-variant norm constrained abstract interpolation problem, which
in the time-invariant case is equivalent to the relaxed commutant
lifting problem \cite{FtHK07b}.

To state the interpolation problem considered in this paper we need
some notation. Throughout $\sU_k$ and $\sY_k$ are Hilbert spaces
with $k$ being an arbitrary integer, and the symbols $\asU$ and
$\asY$ stand for the Hilbert direct sums $\oplus_{k\in \BZ}\, \sU_k$
and $\oplus_{k\in \BZ}\, \sY_k$, respectively. We shall consider
operator matrices $H=\left[H_{j,\,k}\right]_{{j,\,k}\in\BZ}$ of
which the $(j,\,k)$-th entry $H_{j,\,k}$ is an operator from $\sU_k$
into $\sY_j$. The set of all such operator matrices will be denoted
by $\eM({\asU},{\asY})$. By ${\eUM}({\asU},{\asY})$ we denote the
subset of $\eM({\asU},{\asY})$ consisting of all
$H=\left[H_{j,\,k}\right]_{{j,\,k}\in\BZ}$ that are \emph{upper
triangular}, that is, $H_{j,\,k}=0$ for each $k<j$.

In the present paper we are particularly interested in those
$H=\left[H_{j,\,k}\right]_{{j,\,k}\in\BZ}$ in $\eUM({\asU},{\asY})$
that have the additional property
\begin{equation}
\label{condH}\sum_{j=-\iy}^{k} \|H_{j,\,k}u_k\|^2\leq
c_H\|u_k\|^2,\quad u_k\in \sU_k \quad (k\in \BZ ),
\end{equation}
where $c_H$ is some constant depending on $H$ only. The set of all
such operator matrices is denoted by $\eUM^2({\asU},{\asY})$. We say
that $H$ belongs to $\eUM^2_{\textup{ball}}({\asU},{\asY})$ whenever
the constant $c_H$ can be taken equal to one. Thus an upper
triangular operator matrix $H$ belongs to
$\eUM^2_{\textup{ball}}({\asU},{\asY})$ if and only if for each
$k\in \BZ$ the $k$-th column of $H$ induces a contractive operator
from $\sU_k$ into $\asY=\oplus_{k\in \BZ}\, \sY_k$. The following is
the main problem treated in this paper.

\begin{probl}
\label{mainprobl}Assume that for each $k\in \BZ$ we have given a
subspace $\sF_k$ of $\sU_k$ and a contraction
\begin{equation}
\label{defomk} \o_k=\begin{bmatrix}
\o_{k,\,1}\\
\o_{k,\,2}
\end{bmatrix}:\sF_k\to \begin{bmatrix}
\sY_k\\\sU_{k-1}
\end{bmatrix}.
\end{equation}
Given this data, find  all
$H=\left[H_{j,\,k}\right]_{{j,\,k}\in\BZ}$ in
$\eUM^2_{\textup{ball}}({\asU},{\asY})$ such that for each $k\in
\BZ$ the following interpolation conditions hold:
\begin{equation}\label{intpolcond1}
H_{k,\,k}|_{\sF_k}=\o_{k,\,1}, \quad
H_{j,\,k}|_{\sF_k}=H_{j,\,k-1}\o_{k,\,2}\quad (j,k\in\BZ,j<k).
\end{equation}
\end{probl}

In the time-invariant case, the spaces $\sF_k=\sF$, $\sU_k=\sU$, and
$\sY_k=\sY$ and the contraction $\o_k=\o$ do not depend on $k$, and
the operators $H_{j,k}$ depend only on the difference $j-k$. In this
setting,  the above problem reduces to the function theory problem
considered in the first paragraph of \cite{FtHK07b}. To see this,
note that in this time-invariant setting the operator matrix $H$ can
be identified with the $\sL(\sU,\sY)$-valued function $F_H$,
analytic on the open unit disc $\BD$, given by
\[
F_H(\l)=\sum_{\nu=0}^\iy \l^\nu H_{-\nu}.
\]
Moreover in this case the interpolation condition
 and the norm constraint in Problem \ref{mainprobl} can be restated as
\[
\o_1+\l F_H(\l)\o_2=F_H(\l)|_{\sF}\quad (\l\in \BD)\ands
\sum_{\nu=0}^\iy \|H_{-\nu} u\|^2\leq \|u\|^2\quad (u\in \sU).
\]

For a particular choice of the contractions $\o_k$, Problem
\ref{mainprobl} appears in a natural way in the analysis of the set
of all solutions to the three chain completion problem
\cite{FFGK97a}, \cite{FFGK97b}. Indeed, see Section 4 in
\cite{FFGK97b} or Section XIV.3 in \cite{FFGK98}, where one can find
Problem \ref{mainprobl} with $\o_k$ being an isometry for each $k\in
\BZ$.

To state our first main result some additional notation is needed.
We use the symbol  $\eUM^\iy({\asU},{\asY})$ to denote the set of
all double infinite upper triangular operator matrices $H$ that
induce bounded linear operators from the Hilbert space
$\asU=\oplus_{k\in \BZ}\, \sU_k$ into the Hilbert space
$\asY=\oplus_{k\in \BZ}\, \sY_k$. If this induced operator is a
contraction, then we say that $H$ belongs to
$\eUM^\infty_\tu{ball}(\asU,\asY)$. In particular,
$\eUM^\infty_\tu{ball}(\asU,\asY)\subset\eUM^2_\tu{ball}({\asU},{\asY})$.
We write $\eUM^\infty_0(\asU,\asY)$ and
$\eUM^\infty_\tu{ball,0}(\asU,\asY)$ for the sets of all strictly
upper triangular operator matrices in $\eUM^\infty(\asU,\asY)$ and
$\eUM^\infty_\tu{ball}(\asU,\asY)$, respectively. Finally, when
$\sU_k=\sY_k$ for each $k\in \BZ$, and hence $\asU=\asY$, we shall
always replace the argument $(\asU, \asY)$ by $(\asU)$. Thus
$\eM(\asU)$ stands for $\eM(\asU, \asU)$, and $\eUM(\asU)$ stands
for $\eM(\asU, \asU)$, etc.  We are now ready to state the first
main result.

\begin{thm}
\label{thmmain1}For each $k\in \BZ$ let $\o_k$ be the contraction
given by \eqref{defomk}. Choose
\begin{equation}\label{Z12}
Z^{(1)}=\left[(Z^{(1)}_{j,\,k})\right]_{{j,\,k}\in\BZ} \in
\eUM^\iy(\asU,\asY), \
Z^{(2)}=\left[Z^{(2)}_{j,\,k}\right]_{{j,\,k}\in \BZ}\in
\eUM^\iy_0(\asU),
\end{equation}
such that
\begin{equation}\label{SchurClass}
\|Z^{(1)}u\|^2+\|Z^{(2)}u\|^2\leq \|u\|^2\quad (u\in{\asU})
\end{equation}
and
\begin{equation}
\label{condZ12} Z^{(1)}_{j,\,k}|_{\sF_k}=\left\{\begin{array}{cl}
\o_{k,\,1},&\mbox{if $j=k$},\\
0,&\mbox{if $j\not =k$},
\end{array}\right. \quad Z^{(2)}_{j,\,k}|_{\sF_k}=\left\{\begin{array}{cl}
\o_{k,\,2},&\mbox{if $j=k-1$},\\
0,&\mbox{if $j\not =k-1$}.
\end{array}\right.
\end{equation}
Then
\begin{equation} \label{HZ12}
H=Z^{(1)}(I_\asU- Z^{(2)})^{-1}
\end{equation}
is a well-defined upper triangular operator matrix, $H$ belongs to
 $\eUM^2_\tu{ball}(\asU,\asY)$, and $H$ is a solution of
\textup{Problem} \ref{mainprobl}.
\end{thm}

Let us explain why formula \eqref{HZ12} makes sense.  In general,
for arbitrary infinite operator matrices the usual matrix product is
not defined. The situation is different for upper triangular
matrices. For instance, for $A\in \eUM(\asU,\asY)$ and $B\in
\eUM(\asU)$ the matrix product $AB$ is well-defined and $AB$ belongs
to $\eUM(\asU,\asY)$. Furthermore, with the usual matrix product
$\eUM(\asU)$ is an algebra with the identity matrix $I_{\asU}$ as a
unit and an operator matrix $M=
\left[M_{j,\,k}\right]_{{j,\,k}\in\BZ}\in \eUM(\asU)$ is invertible
in $\eUM(\asU)$ if and only if for each $j\in \BZ$ the $j$-th
diagonal entry $M_{j,\,j}$ is invertible as an operator on $\sU_j$
(see Subsection \ref{ssecinvUM} below for further details). {}From
these remarks is clear that the operator $I_\asU- Z^{(2)}$ in
\eqref{HZ12} is invertible in $\eUM(\asU)$ and that the product in
\eqref{HZ12} is well-defined.

One can always find $Z^{(1)}$ and $Z^{(2)}$ satisfying the
conditions \eqref{Z12}, \eqref{SchurClass} and \eqref{condZ12} in
Theorem \ref{thmmain1}. For instance one can take
\begin{equation}
\label{specchoice} Z^{(1)}_{j,\,k}=\d_ {j,\,k}\o_{k,\,1}\Pi_{\sF_k}
\ands Z^{(2)}_{j,\,k}=\d_ {j,\,k-1}\o_{k,\,2}\Pi_{\sF_k}.
\end{equation}
Here $\d_ {j,\,k}$ is the Kronecker delta and the map
$\Pi_{\sF_{j}}$ is the orthogonal projection of $\sU_j$ onto
$\sF_{j}$. The solution of Problem \ref{mainprobl} corresponding to
this choice of $Z^{(1)}$ and $Z^{(2)}$ is given by
$\tilde{H}=[\tilde{H}_{j,k}]_{{j,\,k}\in\BZ}$ with
\begin{equation*}
\label{centsol} \tilde{H}_{j,k}= \left\{\begin{array}{cl}
0,&j>k,\\
\noalign{\vskip6pt}
\o_{k,\,1}\Pi_{\sF_k}, &j=k,\\
\noalign{\vskip6pt}
(\o_{j,\,1}\Pi_{\sF_j})(\o_{j+1,\,2}\Pi_{\sF_{j+1}}) \cdots
(\o_{k-1,\,2}\Pi_{\sF_{k-1}})(\o_{k,\,2}\Pi_{\sF_{k}}),&j<k.
\end{array}\right.
\end{equation*}
Thus Problem \ref{mainprobl} is always solvable.

Our second main result shows, in particular, that the method of
Theorem~\ref{thmmain1} gives all solutions to Problem
\ref{mainprobl}, that is, given a solution $H$  to Problem
\ref{mainprobl}, there exists
 a pair of operator matrices $(Z^{(1)}, Z^{(2)})$ satisfying
\eqref{Z12}, \eqref{SchurClass} and \eqref{condZ12} such that $H$ is
given by the formula \eqref{HZ12}. In general, such a pair
$(Z^{(1)}, Z^{(2)})$ is not uniquely determined by $H$. This
phenomenon already appears in the time-invariant case and can be
illustrated by simple examples. For instance,  assume all spaces
$\sF_k$ consist of the zero element only. In that case $H=0$ is in
$\eUM^2_\tu{ball}(\asU,\asY)$ and is a solution, while \eqref{HZ12}
holds with $Z^{(1)}=0$ and with any $Z^{(2)}$ from
$\eUM^\iy_0(\asU)$.

Given a solution $H$ to Problem \ref{mainprobl}, we shall describe
the set of all pairs $(Z^{(1)}, Z^{(2)})$ satisfying \eqref{Z12},
\eqref{SchurClass} and \eqref{condZ12} stated in the above theorem
such that $H$ is given by \eqref{HZ12}. The precise result is given
by Theorem \ref{thmmain2full} in Section \ref{sec:proof3}. Here we
only describe some of the main ingredients entering into the proof
and present an abbreviated version of this theorem.


Let $H$ be a  solution to Problem \ref{mainprobl}, and let $H_k$ be
the $k$-th column of $H$. Recall that $H_k$ defines a contraction
from $\sU_k$ into $\asY$. Let $D_{H_k}=(I_{\sU_k}-(H_k)^*H_k)^{1/2}$
denote  the corresponding defect operator,  and let $\sD_{H_k}$ be
the corresponding defect space, i.e., $\sD_{H_k}$ is the closure of
the range of $D_{H_k}$ in $\sU_k$. Since \eqref{intpolcond1} is
satisfied, for each $f\in \sF_k$ we have
\begin{eqnarray*}
\|D_{H_k}f\|^2 &=&\|f\|^2-\|H_kf\|^2
=\|f\|^2-\|\o_{k,\,1}f\|^2-\|H_{k-1}\o_{k,\,2}f\|^2\\
&=&\|f\|^2-\|\o_{k,\,1}f\|^2-\|\o_{k,\,2}f\|^2+\|\o_{k,\,2}f\|^2-\|H_{k-1}\o_{k,\,2}f\|^2\\
&=&\|D_{\o_k}f\|^2+\|D_{H_{k-1}}\o_{k,\,2}f\|^2.
\end{eqnarray*}
Hence we can define a contraction $\o_{H_k}$ by
\begin{equation}\label{defomHk}
\o_{H_k}:\sF_{H_k}:=\overline{D_{H_k}\sF_k}\to\sD_{H_{k-1}},\quad
\o_{H_k}D_{H_k}|_{\sF_k}=D_{H_{k-1}}\o_{k,\,2}.
\end{equation}
Now put $\asD_{H}=\oplus_{k\in \BZ}\sD_{H_k}$, and let $\eC_{H,\,
\o}$ be the set of all operator matrices
$C=\left[C_{j,\,k}\right]_{j,k\in\BZ}$ in
$\eUM^\infty_\tu{ball,0}(\asD_H)$ such that
\begin{equation}\label{KH}
C_{j,\,k}|_{\sF_{H_k}}=\left\{\begin{array}{cl}
\o_{H_k},&\mbox{if $j=k-1$},\\
0,&\mbox{if $j\not=k-1$}.
\end{array}\right.
\end{equation}
Thus
\begin{eqnarray} \label{paramCH}
&&\eC_{H,\, \o}=\{C\in \eUM^\infty_\tu{ball,0}(\asD_H) \mid
\mbox{the $(j,k)$-th entry $C_{j,\,k}$}\nonumber\\
\noalign{\vskip6pt} &&\hspace{3.5cm}\mbox{of $C$ satisfies
\eqref{KH} for each $j,k\in\BZ$}\}.
\end{eqnarray}
We can now state  the abbreviated version of our second main result.

\begin{thm}
\label{thmmain2} Let $H$ be a solution to \textup{Problem}
\ref{mainprobl}. Then there exists a pair of operator matrices
$(Z^{(1)}, Z^{(2)})$ satisfying \eqref{Z12}, \eqref{SchurClass} and
\eqref{condZ12} such that $H$ is given by \eqref{HZ12}. Furthermore,
the set of all such pairs $(Z^{(1)}, Z^{(2)})$ is in one-to-one
correspondence with the set $\eC_{H,\, \o}$.
\end{thm}
\noindent The full version of the above theorem (see Theorem
\ref{thmmain2full} below) will also present necessary and sufficient
conditions  guaranteeing that the set $\eC_{H,\, \o}$ consists of a
single element only.

\smallskip
Let $H$ be a solution to \textup{Problem} \ref{mainprobl}. In the
analysis of the set of all pairs $(Z^{(1)}, Z^{(2)})$ satisfying
\eqref{Z12}--\eqref{HZ12}  the following problem enters in a natural
way.

\begin{probl}
\label{probl2} Given $H\in \eUM^2_\tu{ball}(\asU,\asY)$, describe
the set of operator matrices $F$ in $\eUM(\asU)$ satisfying
\begin{equation}\label{major}
F+F^*\geq H^*H+I_\asU \ands F_{j,\,j}=I_{\sU_j}\quad (j\in\BZ).
\end{equation}
\end{probl}

Note that the matrix product $H^*H$ is well-defined because each
column of $H$ induces a contractive operator (see Subsection
\ref{ssecHstH} for further details). The inequality sign in
\eqref{major} means that the operator matrix
$\half(F+F^*)-H^*H-I_\asU$ is non-negative (see Subsection
\ref{ssecherm} for the definition of this notion and further
details). The fact that \eqref{major} appears in the analysis,
follows from the observation that $F=(I_\asU-Z^{(2)})^{-1}$
satisfies \eqref{major} whenever the pair $(Z^{(1)}, Z^{(2)})$
satisfies the conditions \eqref{Z12}--\eqref{HZ12}. This connection
will be made more precise in Theorem \ref{th:UM2}. The solution to
Problem \ref{probl2} will be obtained as a corollary to Theorem
\ref{th:harmaj}.

For the time-invariant case Theorems \ref{thmmain1} and Theorem
\ref{thmmain2} can be found in \cite{FtHK07b}. By using the
reduction techniques developed in Chapter X of \cite{FFGK98} (also
\cite{FFGK96}) Problem \ref{mainprobl} can be transformed into a
problem of the type considered in \cite{FtHK07b}. This
transformation together with techniques from \cite{FFGK98} can be
used to present an alternative way to prove our main results. We
shall not develop this approach in the present paper. Theorems
\ref{thmmain1} and Theorem \ref{thmmain2} can be used to solve a
time-variant analogue of the relaxed commutant lifting problem. We
will describe this connection in the final section of the paper.

This paper consists of five sections not counting this introduction.
The first section has a preliminary character. We introduce some
additional notation and recall a  number of elementary facts about
operator matrices that will be used in the proofs. In Section
\ref{sec:balUM2} we outline a general approach to deal with Problem
\ref{mainprobl} and prove Theorem \ref{thmmain1}. Section
\ref{sec:problem2} is divided into three subsections. In this
section a time-variant analogue of the Cayley transform is used to
relate operator matrices from $\eUM^\infty_\tu{ball,0}$ to positive
real operator matrices from $\eUM(\asU)$. We apply this result to
solve Problem \ref{probl2} and to parameterize the set of all its
solutions. Theorem \ref{thmmain2} is proved in Section
\ref{sec:proof3}; this section also presents the full version of
Theorem \ref{thmmain2} and its proof. Here we also discuss the
problem of finding necessary and sufficient conditions for the
existence of a unique solution to Problem \ref{mainprobl}. In the
final section an example involving finite operator matrices will be
presented and we discuss the connection with a time-variant analogue
of the relaxed commutant lifting problem.

\section{Preliminaries}\label{sec:prelim}\setcounter{equation}{0}
In this section we bring together a  number of elementary facts
about operator matrices that will be used in the sequel. In what
follows  we assume the reader to be familiar with the notations
introduced in the previous section.

\subsection{The set $\eM({\asU},{\asY})$ and bounded operators}
\label{ssecM} The set $\eM({\asU},{\asY})$ is a linear space with
respect to the usual operation of matrix addition. Given
$M=\left[M_{j,\,k}\right]_{{j,\,k}\in\BZ}$ in $\eM({\asU},{\asY})$
and $j\leq k$ we write $\Delta_{j,\,k}(M)$ for
 the $\{j,k\}$-\emph{finite section} of $M$, that is,
\begin{equation}\label{sectionM2}
\Delta_{j,\,k}(M)= \begin{bmatrix}
M_{j,\,j}&\cdots&M_{j,\,k}\\
\vdots&&\vdots\\
M_{k,\,j}&\cdots&M_{k,\,k}
\end{bmatrix}.
\end{equation}
Note that $\Delta_{j,\,k}(M)$ defines a bounded linear operator from
$\sU_j\oplus\cdots\oplus\sU_k$ into $\sY_j\oplus\cdots\oplus\sY_k$.

In general, an operator matrix $M\in\eM({\asU},{\asY})$ does not
induce in a canonical way  a bounded operator from
$\asU=\oplus_{k\in \BZ}\, \sU_k$ into the space $\asY=\oplus_{k\in
\BZ}\, \sY_k$. In order for this to happen it is necessary and
sufficient that
\begin{equation}
\label{normM} \sup_{j\leq k}\|\Delta_{j,\,k}(M)\|<\iy.
\end{equation}
Furthermore, if  \eqref{normM} is satisfied, then the quantity in
the left hand side of \eqref{normM} is equal to the  norm of $M
=\left[M_{j,\,k}\right]_{{j,\,k}\in\BZ}$ as an operator from $\asU$
into $\asY$.

\subsection{Invertibility in the algebra $\eUM({\asU})$} \label{ssecinvUM}
Let $\asX=\oplus_{k\in \BZ}\, \sX_k$, $\asU=\oplus_{k\in \BZ}\,
\sU_k$, and $\asY=\oplus_{k\in \BZ}\, \sY_k$ be Hilbert space direct
sums. If $B\in \eUM(\asX,\asU)$ and $A\in  \eUM(\asU,\asY)$, then
the (block) matrix product $AB$ is well-defined and
$AB\in\eUM(\asX,\asY)$. Moreover, for $C\in\eUM(\asX,\asY)$ we have
\begin{equation}\label{prodABC}
AB=C \quad\Longleftrightarrow\quad
\Delta_{j,\,k}(A)\Delta_{j,\,k}(B)= \Delta_{j,\,k}(C)\ ( j\leq k).
\end{equation}
In particular, the set $\eUM({\asU})$ is closed under the usual
multiplication of matrices. In fact, from \eqref{prodABC} we see
that $\eUM({\asU})$ is an algebra with the identity matrix
$I_{\asU}$ as a unit. {}From \eqref{prodABC} it also follows that
the operator matrix $M=\left[M_{i,\,j}\right]_{i,\,j\in\BZ}\in
\eUM({\asU})$ is invertible in $\eUM({\asU})$ if and only if for
each $j\in\BZ$ the $j$-th diagonal entry $M_{j,\,j}$ is an
invertible operator on $\sU_j$. In that case, we have
\begin{equation}\label{FinInv}
\de_{j,\,k}(M^{-1})=\de_{j,\,k}(M)^{-1}\quad (j,k\in\BZ, j\leq k).
\end{equation}
In particular, the $(j,j)$-th entry of $M^{-1}$ is equal to
$M_{k,\,k}^{-1}$.

The above mentioned properties of $\eUM({\asU})$ also follow from
the fact that  an operator matrix from $\eUM({\asU})$ can be
identified in the usual way with a linear transformation on the
linear space $\asU^+$. By definition, the space $\asU^+$ consists of
all double infinite one column matrices $\mathbf{u}=[u_j\,]_{j\in
\BZ}$, with $u_j\in\sU_j$ for each $j\in \BZ$, such that $u_\nu=0$
for $\nu
> \ell$, for some $\ell$ depending on $\mathbf{u}$.

\subsection{Hermitian and non-negative operator matrices.}
\label{ssecherm} An operator matrix $M\in \eM({\asU})$,
$M=\left[M_{j,\,k}\right]_{{j,\,k}\in\BZ}$, is said to be
\emph{hermitian} if $M^*=M$, where $M^*$ is the operator matrix
$M^*=\left[(M_{k,\,j})^*\right]_{{j,\,k}\in\BZ}$. The \emph{real
part} of $M\in \eM({\asU})$ is the operator matrix $\re M$ given by
\begin{equation}
\label{realpart} \re M=\half(M+M^*).
\end{equation}
Obviously, $\re M$ is hermitian. We call $M\in \eM({\asU})$
\emph{non-negative} if for each ${j,\,k}\in\BZ$, $j\leq k$, the
finite section $\Delta_{j,\,k}(M)$ induces a non-negative operator
on the Hilbert space direct sum $\sU_j\oplus\cdots\oplus\sU_k$. In
that case $M$ is hermitian. For operator matrices $M$ and $N$ in
$\eM({\asU})$ we say that $M$ is {\em greater than or equal to} $N$,
and write $M\geq N$, if the operator matrix $M-N$ is non-negative.
Hence $M\geq 0$ means that $M$ is non-negative. Finally, an operator
matrix $M\in \eUM(\asU)$ is said to be \emph{positive real} whenever
$\re M$ is non-negative. Positive real operator matrices $M\in
\eUM(\asU)$ that induce bounded operators on $\asU=\oplus_{k\in
\BZ}\, \sU_k$ (i.e., $M\in \eUM^\iy(\asU)$) have been extensively
studied in \cite{ADP99, ADP01}.

\subsection{The operator matrix $H^*H$.}\label{ssecHstH} Let $H\in
\eUM^2_\tu{ball}(\asU,\asY)$. Recall that for each $k\in \BZ$ the
$k$-th column $H_k$ of $H$ induces a contractive operator, also
denoted by $H_k$, from $\sU_k$ into $\asY=\oplus_{k\in \BZ}\,
\sY_k$. It follows that for each $j$ and $k$ in $\BZ$ the product
$(H_j)^*H_k$ is a well-defined contraction from $\sU_k$ into
$\sY_j$. We define $H^*H$ to be the operator matrix in $\eM(\asU)$
given by
\begin{equation}\label{HstH}
H^*H=\mat{c}{(H_j)^*H_k}_{j,\,k\in\BZ}.
\end{equation}
Clearly, $H^*H\in \eM(\asU)$ is hermitian. In fact, since for each
${j,\,k}\in\BZ$, $j\leq k$,
\[
\Delta_{j,\,k}(H^*H)=\begin{bmatrix} (H_j)^*\\\vdots\\(H_k)^*
\end{bmatrix}\begin{bmatrix} H_j& \cdots& H_k
\end{bmatrix},
\]
the operator matrix $H^*H$ is non-negative. Note that $H^*H$ is the
real part of the operator matrix $V\in\eUM(\asU)$ be given by
\begin{equation}
\label{defV} V=\mat{c}{V_{j,k}}_{j,k\in\BZ},\quad
V_{j,\,k}=\left\{\begin{array}{cl}
2(H_j)^*H_k, & \mbox{for $j<k$}, \\
\noalign{\vskip6pt}
(H_j)^*H_j, & \mbox{for $j=k$},\\
\noalign{\vskip6pt} 0, &\mbox{for $j>k$}.
\end{array}\right.
\end{equation}

\section{The set $\eUM^2_\tu{ball}(\asU,\asY)$ and the
proof of Theorem
\ref{thmmain1}}\label{sec:balUM2}\setcounter{equation}{0} The main
result  of this section (Theorem \ref{th:UM2} below) shows how
elements in $\eUM^2_\tu{ball}(\asU,\asY)$ can be constructed from a
pair of operators $Z^{(1)}\in\eUM^\infty(\asU,\asY)$ and
$Z^{(2)}\in\eUM_0^\infty(\asU)$ satisfying an additional norm
constraint. This result together with Proposition \ref{pr:UM2intpol}
allows us to prove Theorem \ref{thmmain1}. Theorem \ref{th:UM2} also
shows the relevance of Problem \ref{probl2} in the analysis of
Problem \ref{mainprobl}.

\begin{thm}\label{th:UM2}
Assume we have given operator matrices
\begin{equation}
\label{Z12b} Z^{(1)}\in\eUM^\infty(\asU,\asY), \quad
Z^{(2)}\in\eUM_0^\infty(\asU)
\end{equation}
satisfying the norm constraint
\begin{equation}\label{Schur2}
\|Z^{(1)}u\|^2+\|Z^{(2)}u\|^2\leq \|u\|^2\quad (u\in{\asU}).
\end{equation}
Then
\begin{equation}\label{ZtoHandF}
H=Z^{(1)} (I_\asU- Z^{(2)} )^{-1}\ands F=(I_\asU-Z^{(2)})^{-1}
\end{equation}
are well-defined operator matrices,
\begin{equation} \label{HFb} H\in\eUM^2_\tu{ball}(\asU,\asY),
\quad F\in\eUM(\asU)
\end{equation} and
\begin{equation}\label{HFconditions}
F+F^*\geq H^*H+I_\asU, \quad  F_{j,\,j}=I_{\sU_j}\quad  (j\in \BZ).
\end{equation}
Conversely, if we have given $H$ and $F$ as in \eqref{HFb}  such
that \eqref{HFconditions} holds, then $F$ is an invertible element
in $\eUM(\asU)$, the operator matrices
\begin{equation}\label{Z1Z2}
Z^{(1)}= HF^{-1} \ands Z^{(2)}= I_{{\asU}}- F^{-1}
\end{equation}
are well-defined and these operator matrices satisfy \eqref{Z12b}
and \eqref{Schur2}. Moreover, the map $(Z^{(1)},Z^{(2)})\mapsto
(H,F)$ defined by \eqref{ZtoHandF} is a one-to-one map from the set
of all pairs $(Z^{(1)},Z^{(2)})$ satisfying \eqref{Z12b} and
\eqref{Schur2} onto the set of all pairs $(H,F)$ satisfying
\eqref{HFb} and  \eqref{HFconditions}. The inverse of this map is
given by the map $(H,F)\mapsto (Z^{(1)},Z^{(2)})$ defined by
\eqref{Z1Z2}.
\end{thm}
\bpr We split the proof into four parts. In the first two parts
$Z^{(1)}$ and $Z^{(2)}$ are given and satisfy \eqref{Z12b} and
\eqref{Schur2}, and we prove that $H$ and $F$ in \eqref{ZtoHandF}
are well-defined and satisfy \eqref{HFb} and \eqref{HFconditions}.
In the third part we prove the reverse statement. In the final part
we show that the maps $(Z^{(1)},Z^{(2)})\mapsto (H,F)$ and
$(H,F)\mapsto(Z^{(1)},Z^{(2)})$ in Theorem \ref{th:UM2} are each
others inverses.

\smallskip\noindent\textit{Part 1.} Assume that $Z^{(1)}\in\eUM^\infty(\asU,\asY)$ and
$Z^{(2)}\in\eUM_0^\infty(\asU)$ satisfy (\ref{Schur2}), and let $H$
be given by the first part of (\ref{ZtoHandF}). Our aim is to prove
that $H\in\eUM^2_\tu{ball}(\asU,\asY)$.

Fix $k,j\in\BZ, j<k$, and $u_k\in\sU_k$. Define
\begin{equation*}
v=\mat{c}{v_j\\\vdots\\v_k}=\de_{j,\,k}\left((I_\asU-Z^{(2)})^{-1}\right)u,\quad
\text{where }u=\mat{c}{0\\\vdots\\0\\u_k}\in\oplus_{n=j}^k\sU_n.
\end{equation*}
Note that $v\in\oplus_{n=j}^k\sU_n$. Since $Z^{(2)}$ is strictly
upper triangular, we have $v_k=u_k$. Observe that
\begin{eqnarray*}
\de_{j,\,k}(Z^{(2)})\de_{j,\,k}\left((I_\asU-Z^{(2)})^{-1}\right)
&=&\de_{j,\,k}\left(Z^{(2)}(I_\asU-Z^{(2)})^{-1}\right)\\
&=&\de_{j,\,k}\left((I_\asU-Z^{(2)})^{-1}-I_\asU\right)\\
&=&\de_{j,\,k}\left((I_\asU-Z^{(2)})^{-1}\right)-I_{\oplus_{n=j}^k\sU_n}
\end{eqnarray*}
and
\[
\de_{j,\,k}(Z^{(1)})\de_{j,\,k}\left((I_\asU-Z^{(2)})^{-1}\right)=
\de_{j,\,k}\left(Z^{(1)}(I_\asU-Z^{(2)})^{-1}\right)=
\de_{j,\,k}(H).
\]
Hence
\[
\de_{j,\,k}(Z^{(2)})v=v-u=:\a\ands
\de_{j,\,k}(Z^{(1)})v=\de_{j,\,k}(H)u=:\b,
\]
where, using $v_k=u_k$,
\[
\a=\mat{c}{v_j\\\vdots\\v_{k-1}\\0}\ands\b=\mat{c}{H_{j,\,k}u_k\\\vdots\\H_{k-1,\,k}u_k\\H_{k,\,k}u_k}.
\]
The assumption (\ref{Schur2}) implies that
\[
\mat{c}{\de_{j,\,k}(Z^{(1)})\\\de_{j,\,k}(Z^{(2)})}
\]
is a contraction. Thus
\begin{equation*}
\sum_{\nu=j}^{k-1} \|v_\nu\|^2+\sum_{\nu=j}^{k}
\|H_{\nu,\,k}u_k\|^2=\|\a\|^2+\|\b\|^2\leq \|v\|^2=\sum_{\nu=j}^{k}
\|v_\nu\|^2.
\end{equation*}
By comparing the first term in the left hand side of the above
inequality with the term in the right side and using $v_k=u_k$, we
see that $\sum_{\nu=j}^{k} \|H_{\nu,\,k}u_k\|^2$ is less than or
equal to $ \|u_k\|^2$. This holds for each $j\leq k$ and
$u_k\in\sU_k$, and therefore the operator defined by the $k$-th
column of $H$ is  a contraction. Since $k\in \BZ$ is arbitrary,
$H\in\eUM^2_\tu{ball}(\asU,\asY)$.

\smallskip\noindent\textit{Part 2.} Under the same assumptions as in Part 1,
let $H$ and $F$ be given by (\ref{ZtoHandF}). Our aim is to prove
that (\ref{HFconditions}) holds. {}From the definition of $F$ the
right hand side of (\ref{HFconditions}) is clear. Thus we have to
prove the inequality in the left hand side of (\ref{HFconditions}).

First assume that there exists an $N\geq0$ such that
$\sU_k=\sY_k=\{0\}$ for $|k|>N$. In that case $\eM(\asU)=\sL(\asU)$,
$\eUM(\asU)=\eUM^\infty(\asU)$ and
$\eUM(\asU,\asY)=\eUM^\infty(\asU,\asY)$. In particular, $H$ and $F$
are bounded operators. We have
\begin{eqnarray*}
I-Z^{(2)*}Z^{(2)}
&=&I-(I-I+Z^{(2)*})(I-I+Z^{(2)})\\
&=&I-(I-F^{-*})(I-F^{-1})\\
&=&F^{-*}+F^{-1}-F^{-*}F^{-1}.
\end{eqnarray*}
Thus
\begin{eqnarray*}
H^*H&=&F^*Z^{(1)*}Z^{(1)}F \leq F^*(I-Z^{(2)*}Z^{(2)})F\\
&=&F^*(F^{-*}+F^{-1}-F^{-*}F^{-1})F =F+F^*-I.
\end{eqnarray*}
So (\ref{HFconditions}) holds in this case.

Now fix $j\leq k$. For $N\geq |j|,|k|$ set
\[
\asU_N=\oplus_{i=-N}^N\sU_i,\quad \asY_N=\oplus_{i=-N}^N\sY_i,
\]
and notice that
\begin{eqnarray*}
&&\de_{-N,\,N}(F)=(I_{\asU_N}-\de_{-N,\,N}(Z^{(2)}))^{-1}\quad\text{and}\\
&&\de_{-N,\,N}(H)=\de_{-N,\,N}(Z^{(1)})(I_{\asU_N}-\de_{-N,\,N}(Z^{(2)}))^{-1}.
\end{eqnarray*}
Next apply  the result of the second paragraph of this part to
\[
\de_{-N,\,N}(Z^{(1)})\in\eUM(\asU_N,\asY_N)\quad\text{and}\quad
\de_{-N,\,N}(Z^{(2)})\in\eUM_0(\asU_N).
\]
Using that $\de_{j,\,k}(\de_{-N,\,N}(M))=\de_{j,\,k}(M)$ for each
$M\in\eUM(\asU)$, it follows that
\begin{equation}\label{HFconditionsFin}
\de_{j,\,k}(F)+\de_{j,\,k}(F^*)\geq \de_{j,\,k}(H^*Q_NH)+I,
\end{equation}
where $Q_N\in\sL(\asY)$ is the orthogonal projection on $\asY_N$.
Since
\[
\de_{j,\,k}(H^*Q_NH)=\mat{c}{(H_j)^*\\\vdots\\(H_k)^*}Q_n\mat{ccc}{H_j&\cdots&H_k},
\]
$\mat{ccc}{H_j&\cdots&H_k}$ is a bounded linear operator, $Q_N\to
I_{\asY}$ as $N\to\infty$ with convergence in the strong operator
topology, and (\ref{HFconditionsFin}) holds for each $N\geq|j|,|k|$,
it follows that
\begin{equation}
\label{sectineq} \de_{j,\,k}(F)+\de_{j,\,k}(F^*)\geq
\de_{j,\,k}(H^*H)+I.
\end{equation}
Thus (\ref{HFconditions}) holds.

\smallskip\noindent\textit{Part 3.} In this part we assume that
$H\in \eUM^2_{\textup{ball}}({\asU},{\asY})$,  $F\in \eUM(\asU)$,
and that condition (\ref{HFconditions}) is fulfilled. We show that
the operator matrices $Z^{(1)}$ and $Z^{(2)}$ defined by
(\ref{Z1Z2}) are in $\eUM^\infty(\asU,\asY)$ and
$\eUM_0^\infty(\asU)$, respectively, and that (\ref{Schur2}) is
satisfied.

The second part of (\ref{HFconditions}) implies that $F$ is
invertible in $\eUM(\asU)$. Thus the operator matrices $Z^{(1)}$ and
$Z^{(2)}$ in (\ref{Z1Z2}) are well defined. Moreover, for each $j\in
\BZ$ the $j$-th  diagonal entry of $F^{-1}$ is the identity operator
$I_{\sU_j}$, and thus the matrix $Z^{(2)}$ is strictly upper
triangular. It remains to show that $Z^{(1)}$ and $Z^{(2)}$ satisfy
(\ref{Schur2}), since this automatically implies that
$Z^{(1)}\in\eUM^\infty(\asU,\asY)$ and
$Z^{(2)}\in\eUM_0^\infty(\asU)$.

First note that it suffices to show that
\begin{equation}\label{FinSchurClass}
\de_{j,\,k}(Z^{(1)})^*\de_{j,\,k}(Z^{(1)})+\de_{j,\,k}(Z^{(2)})^*\de_{j,\,k}(Z^{(2)})\leq
\de_{j,\,k}(I_\asU) \quad (j\leq k).
\end{equation}
Indeed, if the above inequalities have been established, then we can
use the results reviewed in the second part of Subsection
\ref{ssecM} to derive (\ref{Schur2}). Fix $j\leq k$, and write $\de$
in place of $\de_{j,\,k}$. {}From $Z^{(2)}= I_\asU-F^{-1}$, we see
that $\de(Z^{(2)})= \de(I_\asU)- \de(F^{-1})$. Now use formula
(\ref{FinInv}) to see that $\de(F^{-1})=\de(F)^{-1}$. But then
\begin{eqnarray*}
&&\de(Z^{(2)})^*\de(Z^{(2)})= \\
\noalign{\vskip4pt}
&&\hspace{1cm}=\left(\de(I_\asU)-\de(F)^{-*}\right)
\left(\de(I_\asU)-\de(F)^{-1}\right)\\
\noalign{\vskip4pt}
&&\hspace{1cm}=\de(I_\asU)- \de(F)^{-1}-\de(F)^{-*}+\de(F)^{-*}\de(F)^{-1}\\
\noalign{\vskip4pt} &&\hspace{1cm}=\de(F)^{-*}
\big(\de(F)^*\de(F)-\de(F)^*- \de(F)+ \de(I_\asU) \big) \de(F)^{-1}.
\end{eqnarray*}
Thus
\begin{eqnarray*}
&&\de(F)^*\left(\de(I_\asU)-\de(Z^{(2)})^*\de(Z^{(2)})\right)\de(F)=\\
\noalign{\vskip4pt} &&\hspace{3cm}=\de(F)^*+\de(F)-\de(I_\asU)\geq
\de(H^*H).
\end{eqnarray*}
Here we used that the first part of (\ref{HFconditions}) implies
(\ref{sectineq}).

Next we consider the equality $H=Z^{(1)}F$. Again using
(\ref{prodABC}), with $\de=\de_{j,\,k}$, we have
$\de(H)=\de(Z^{(1)})\de(F)$, and thus
\begin{equation*}
\de(F)^*\de(Z^{(1)})^*\de(Z^{(1)})\de(F)=\de(H)^*\de(H).
\end{equation*}
By combining this with the result of the previous paragraph we see
that
\begin{eqnarray*}
&&\de(F)^*\big(\de(I_\asU)- \de(Z^{(1)})^*\de(Z^{(1)})-
\de(Z^{(2)})^*\de(Z^{(2)})\big)\de(F)\geq\\
\noalign{\vskip6pt} &&\hspace{2cm}\geq
\de(H^*H)-\de(H)^*\de(H)\geq0.
\end{eqnarray*}
To see that the last inequality holds, let $P$ denote the projection
from ${\asY}$ onto $\sY_j\oplus\cdots\oplus\sY_{k}$ and observe that
\begin{eqnarray*}
\de(H^*H)&=&\begin{bmatrix}
(H_j)^*H_j&\cdots&(H_j)^*H_k\\
\vdots&&\vdots\\
(H_k)^*H_j&\cdots&(H_k)^*H_k
\end{bmatrix}= \begin{bmatrix}
(H_j)^*\\
\vdots\\
(H_k)^*
\end{bmatrix}  \begin{bmatrix}
H_j&\cdots&H_k
\end{bmatrix}\\
&\geq&
\begin{bmatrix}
(H_j)^*\\
\vdots\\
(H_k)^*
\end{bmatrix} P \begin{bmatrix}
H_j&\cdots&H_k
\end{bmatrix}
=\de(H)^*\de(H).
\end{eqnarray*}
Since $\de(F)$ is invertible, we obtain (\ref{FinSchurClass}).

\smallskip\noindent\textit{Part 4.}
In case $Z^{(1)}$ and $Z^{(2)}$ satisfy the assumptions of Parts 1
and 2 and $H$ and $F$ are given by (\ref{ZtoHandF}), it is clear
that
\[
HF^{-1}=Z^{(1)}\ands I-F^{-1}=Z^{(2)}.
\]
If $H$ and $F$ satisfy the assumptions of Parts 3 and $Z^{(1)}$ and
$Z^{(2)}$ are given by (\ref{Z1Z2}), then
\[
(I-Z^{(2)})^{-1}=(I_\asU-I_\asU+F^{-1})^{-1}=F\ands
Z^{(1)}(I_\asU-Z^{(2)})^{-1}=HF^{-1}F=H.
\]
Thus the maps $(Z^{(1)},Z^{(2)})\mapsto (H,F)$ and
$(H,F)\mapsto(Z^{(1)},Z^{(2)})$ in Theorem \ref{th:UM2} are each
others inverses. \epr

\medskip
The next proposition is an addition to Theorem \ref{th:UM2}. It
takes into account the interpolation condition on $H$ in
\eqref{intpolcond1} and on the pair $(Z^{(1)}, Z^{(2)})$ in
\eqref{condZ12}.

\begin{prop}\label{pr:UM2intpol}
Let $Z^{(1)}$ and $Z^{(2)}$ be as in \eqref{Z12b} and
\eqref{Schur2}, and define $H$ and $F$ by \eqref{ZtoHandF}. Then the
pair $(Z^{(1)},Z^{(2)})$ satisfies the interpolation conditions
\eqref{condZ12} if and only if $H$ satisfies the interpolation
condition \eqref{intpolcond1} and
\begin{equation}\label{intpolcondF}
F_{j,\,k}|_{\sF_k}=F_{j,\,k-1}\o_{k,2}\quad(j,k\in\BZ,j<k).
\end{equation}
\end{prop}
\bpr Let $Z^{(1)}$ and $Z^{(2)}$ satisfy \eqref{Z12b} and
\eqref{Schur2}, and let $H$ and $F$ be defined by \eqref{ZtoHandF}.

We begin with a general remark about the interpolation conditions
\eqref{intpolcond1}, \eqref{condZ12}, and \eqref{intpolcondF}.
Recall that for each $k\in \BZ$ the space $\sF_k$ is a subspace of
$\sU_k$. In what follows $\t_k$ is the canonical embedding of
$\sF_k$ into $\sU_k$. Furthermore, $\asF$ will denote the Hilbert
direct sum $\oplus_{k\in \BZ}\, \sF_k$. Now let
\begin{equation}
\label{defEOm12} E=\left[E_{j,\,k}\right]_{j,\,k\in \BZ}, \quad
\om^{(1)}=\left[\om^{(1)}_{j,\,k}\right]_{j,\,k\in \BZ},\quad
\om^{(2)}=\left[\om^{(2)}_{j,\,k}\right]_{j,\,k\in \BZ},
\end{equation}
be the operator matrices defined by
\begin{equation}
\label{defE} E_{j,\,k}=\left\{\begin{array}{cl}
\t_k,&\mbox{if $j=k$},\\
0,&\mbox{if $j\not =k$},
\end{array}\right.
\end{equation}
and
\begin{equation}
\label{defOm12} \om^{(1)}_{j,\,k}=\left\{\begin{array}{cl}
\o_{k,\,1},&\mbox{if $j=k$},\\
0,&\mbox{if $j\not =k$},
\end{array}\right. \quad \om^{(2)}_{j,\,k}=\left\{\begin{array}{cl}
\o_{k,\,2},&\mbox{if $j=k-1$},\\
0,&\mbox{if $j\not =k-1$}.
\end{array}\right.
\end{equation}
Observe that both $E$ and $\om^{(1)}$ are diagonal operator
matrices, $E\in \eUM^\infty_\tu{ball}(\asF, \asU)$ and $\om^{(1)}\in
\eUM^\infty_\tu{ball}(\asF, \asY)$, while $\om^{(2)}\in
\eUM^\infty_\tu{ball,\,0}(\asF, \asU)$ is a shifted diagonal
operator, that is, all the entries of $\om^{(2)}$  are zero except
those in the first diagonal above the main diagonal. Using these
operator matrices we can restate the interpolation conditions. In
fact, we have
\begin{eqnarray}
&&\eqref{intpolcond1}\hspace{.5cm}\Longleftrightarrow \quad  HE=\om^{(1)} +H\om^{(2)},\label{IP1}\\
&&\eqref{condZ12}\hspace{.2cm}\quad\Longleftrightarrow \quad
Z^{(1)}E=\om^{(1)}\ands Z^{(2)}E=\om^{(2)},\label{IP2}\\
&&\eqref{intpolcondF}\quad\Longleftrightarrow \quad
FE-E=F\om^{(2)}.\label{IP3}
\end{eqnarray}

Now assume that the pair $(Z^{(1)}, Z^{(2)})$ satisfies the
interpolation condition \eqref{condZ12}. Since  $H$ and $F$ are
given by \eqref{ZtoHandF}, we have
\begin{eqnarray*}
&&H=Z^{(1)}(I_\asU-Z^{(2)})^{-1}=Z^{(1)}+Z^{(1)}(I_\asU-Z^{(2)})^{-1}Z^{(2)}=Z^{(1)}+HZ^{(2)},\\
\noalign{\vskip6pt}
&&F=(I_\asU-Z^{(2)})^{-1}=I_\asU+(I_\asU-Z^{(2)})^{-1}Z^{(2)}=I_\asU+FZ^{(2)}.
\end{eqnarray*}
Hence, using \eqref{IP2}, we see that
\[
HE=Z^{(1)}E+HZ^{(2)}E=\om^{(1)}+H\om^{(2)}, \quad
FE=E+FZ^{(2)}E=E+F\om^{(2)}.
\]
But then we can use \eqref{IP1} to conclude that $H$ satisfies
\eqref{intpolcond1}, and we can use \eqref{IP3} to conclude that $F$
satisfies \eqref{intpolcondF}.

Next assume that $H$ and $F$ satisfy the interpolation conditions
\eqref{intpolcond1} and \eqref{intpolcondF}, respectively. {}From
\eqref{ZtoHandF} we see that
\[
Z^{(1)}=HF^{-1} \ands Z^{(2)}=I-F^{-1}.
\]
Note that \eqref{intpolcondF} and  \eqref{IP3} imply that
$F^{-1}E=E-\om^{(2)}$. Hence
\begin{eqnarray*}
&&Z^{(1)}E=HF^{-1}E=HE-H\om^{(2)}=\om^{(1)} \quad\mbox{by
\eqref{IP1}},\\
&&Z^{(2)}E=E-F^{-1}E=\om^{(2)}.
\end{eqnarray*}
But then we can use the equivalence in \eqref{IP2} to conclude that
the pair $(Z^{(1)},Z^{(2)})$ satisfies the interpolation conditions
\eqref{condZ12} as desired. \epr

\medskip
\noindent\textbf{Proof of Theorem \ref{thmmain1}.} Let $Z^{(1)}$ and
$Z^{(2)}$ be a pair of operator matrices satisfying \eqref{Z12},
\eqref{SchurClass} and \eqref{condZ12}. Define $H$ by formula
\eqref{HZ12}. Theorem \ref{th:UM2} tells us that $H$ is well-defined
and $H\in\eUM^2_\tu{ball}(\asU,\asY)$; see formula \eqref{HFb}.
Since $(Z^{(1)},Z^{(2)})$ satisfies the interpolation conditions
\eqref{condZ12}, we see from Proposition \ref{pr:UM2intpol} that $H$
satisfies \eqref{intpolcond1}. Thus $H$ is a solution to Problem
\ref{mainprobl}.\epr

\section{Majorants of $H^*H$ and the solution to Problem \ref{probl2}}\label{sec:problem2}
\setcounter{equation}{0} In this section we solve Problem
\ref{probl2} and give a parametrization of the set of all its
solutions. The first subsection, which has a preliminary character,
deals with a time-variant version of the Cayley transform. The main
result (Theorem \ref{th:harmaj} below) is presented in the second
subsection, which is then used in the final subsection to solve
Problem \ref{probl2}.

\subsection{The Cayley
transform}\label{ssec:Ctransf}Let
$C\in\eUM^\infty_\tu{ball,0}(\asU)$. Then we know from Subsection
\ref{ssecinvUM} that $I_\asU-C$  is invertible in the algebra
$\eUM(\asU)$. It follows that $K$ given by
\begin{equation}\label{Ctransf}
K=(I_\asU+C)(I_\asU-C)^{-1}
\end{equation}
is a well-defined element of $\eUM(\asU)$. We shall refer to $K$ as
the \emph{Cayley transform} of $C$. The following proposition shows
that $K$ is positive real (see \cite{ADP99}, page 94, for a related
but somewhat less general result).

\begin{prop}
\label{prCtransf} The map $C\mapsto K$ defined by \eqref{Ctransf}
establishes a one-to-one correspondence between the set of all
operator matrices $C$ in $\eUM^\infty_\tu{ball,0}(\asU)$ and the
positive real $K\in \eUM(\asU)$ satisfying $K_{j,\,j}=I_{\sU_j}$ for
all $j\in\BZ$.
\end{prop}
\bpr Assume that $K$ is defined by (\ref{Ctransf}) for a
$C\in\eUM^\infty_\tu{ball,0}(\asD)$. Then
\[
\K = I_\asU + 2 C \left(I_\asU-  C \right)^{-1}.
\]
Since $(I_\asU-  C)^{-1}$ is upper triangular and $C$ is strictly
upper triangular, it follows that $K_{j,\,j}=I_{\sU_j}$ for each
$j\in\BZ$. Given an operator matrix $M$ let $\Delta(M) =
\Delta_{j,\,k}(M)$ denote the finite section of $M$ for $j \leq k$.
We have to prove (see Subsection \ref{ssecherm}) that $\Delta(\re
K)$ is non-negative.  To do this note that
\begin{eqnarray*}
2 \Delta(\re \K) &=&    (I -  \Delta(C^* ) )^{-1} (I + \Delta(C^* ))
+  (I + \Delta( C)) (I-  \Delta(C) )^{-1}\\
&=&   (I -  \Delta(C^*  ))^{-1}
\left\{(I + \Delta(C^*  ) )(I-  \Delta(C)) + \right.\\
&& \hspace{1cm}+
\left. (I -  \Delta(C^* ))(I + \Delta( C))\right\}(I-  \Delta(C) )^{-1}\\
&=& 2 (I -  \Delta(C^* ))^{-1} \left\{I - \Delta(C^* )
\Delta(C)\right\}(I-  \Delta(C) )^{-1}.
\end{eqnarray*}
In other words,
\begin{equation}\label{cay9}
 \Delta(\re \K ) = (I -  \Delta(C^* ))^{-1}
\left\{I - \Delta(C )^* \Delta( C)\right\}(I-  \Delta(C) )^{-1}.
\end{equation}
Here $I=I_{ \sU_j\oplus\cdots \oplus\sU_k}$. Because $C$ is a
contraction, $\Delta(C)$ is also a contraction. Hence   all finite
sections $\Delta(\re \K )$ of $\re \K$ are non-negative. Therefore
$\K$ is positive real.

Conversely, for a positive real matrix $\K$ in $\eUM(\asU)$
satisfying $\diag \K_{j,\,j} =I_{\sU_j}$ for $j\in\BZ$, consider the
operator matrix $C$ defined by
\begin{equation}\label{eq:invCayley}
C  = \left(\K  - I_\asU\right)\left(\K + I_\asU \right)^{-1}.
\end{equation}
We know that $K+I_\asU\in \eUM(\asU)$ and for each $j\in \BZ$ the
$j$-th diagonal element of $K+I_\asU$ is $2I_{\sU_j}$. Hence $K+I$
is invertible in $\eUM(\asU)$ (see Subsection \ref{ssecinvUM}).
Moreover, $K-I$ is in $\eUM_0(\asU)$. Thus $C$ in
\eqref{eq:invCayley} is a well defined operator matrix in
$\eUM_0(\asU)$. We claim that $C$ is in
$\eUM^\infty_\tu{ball,0}(\asU)$. To see this let $\Delta$ and $I$ be
as in the previous paragraph. Using \eqref{prodABC} and
\eqref{FinInv} we have
\[
\Delta( C) = \left(\Delta(\K)  - I\right)\left(\Delta(\K) + I
\right)^{-1}.
\]
Thus
\begin{eqnarray*}
I - \Delta( C)^*\Delta( C) &=& I - (\Delta(\K)^* +
I)^{-1}(\Delta(\K)^*  - I)
(\Delta(\K)  - I)(\Delta(\K)+ I)^{-1}\\
&=& (\Delta(\K)^*+ I)^{-1}\left\{(\Delta(\K)^*+ I) (\Delta(\K) +
I)\right.
\\&&\left. \hspace{1cm}- (\Delta(\K)^*  - I)(\Delta(\K)  - I) \right\}(\Delta(\K) + I)^{-1}\\
&=&  2 ( \Delta(\K)^*+ I)^{-1}\left(\Delta(\K)^* +  \Delta(\K) \right)(\Delta(\K) + I)^{-1}\\
&=&  4( \Delta(\K)^*+ I)^{-1}\Delta(\re \K)(\Delta(\K) + I)^{-1}\geq
0.
\end{eqnarray*}
Hence any finite section of $C$ is a contraction. Therefore  $C$ is
a contraction, and thus in $\eUM^\infty_\tu{ball,0}(\asU)$.

Finally, one easily verifies that the maps $C\mapsto K$ given by
(\ref{Ctransf}) and $K\mapsto C$ given by (\ref{eq:invCayley}) are
each others inverses. Hence the map $C\mapsto K$ defined by
\eqref{Ctransf} has the desired properties. \epr

\medskip

If $\K$ in $\eUM(\asU)$ is positive real with $K_{j,\,j}=I_{\sU_j}$
for $j\in\BZ$, then $C$ defined by (\ref{eq:invCayley}) will be
called the \emph{inverse Cayley transform} of $\K$.

\subsection{Time-variant harmonic majorants of $H^*H$}\label{ssec:harmonic}

Let $H\in\eUM^2_\tu{ball}(\asU,\asY)$ be given, and consider the
operator matrix $H^*H$ (see Subsection \ref{ssecHstH}). In the
present subsection we describe the operator matrices
$W\in\eUM(\asU)$ satisfying
\begin{equation}\label{harmaj}
\re W\geq H^*H\ands W_{j,\,j}=I_{\sU_j}\quad (j\in\BZ).
\end{equation}
This description  will be used in the next subsection  to give the
solution to Problem~\ref{probl2}.

When $W\in\eUM(\asU)$ satisfies the first identity in \eqref{harmaj}
we call $\re W$ a \emph{time-variant harmonic majorant} of $H^*H$.
In that case, since $H^*H$ is non-negative, $W$ is automatically
positive real. Time-variant harmonic majorants of $H^*H$ do exist.
In fact (see Subsection \ref{ssecHstH}) the operator matrix $V$
defined by \eqref{defV} belongs to $\eUM(\asU)$ and $\re V=H^*H$.
Thus $H^*H$ is its own time-variant harmonic majorant.

To describe all $W\in \eUM(\asU)$ satisfying \eqref{harmaj} recall
that $\asD_H$ is the Hilbert space direct sum
$\oplus_{k\in\BZ}\sD_{H_k}$, where $H_k$ is the $k$-th column of $H$
and $\sD_{H_k}$ is the corresponding defect space. The latter space
is well-defined because $H_k$ defines a contraction from $\sU_k$ in
to $\asY$. We define $\nabla_H$ and $\Pi_H$ to be the diagonal
operator matrices in $\eUM(\asD_H)$ and $\eUM(\asU,\asD_H)$,
respectively, given by
\begin{equation}\label{defNaPi}
(\nabla_H)_{j,\,k}=\left\{\begin{array}{cl} D_{H_k}&\mbox{for $j=k$},\\
0&\mbox{for $j\not=k$}.
\end{array}\right.\quad \mbox{and}\quad  (\Pi_H)_{j,\,k}=\left\{\begin{array}{cl}
\Pi_{H_k}&\mbox{for $j=k$},\\
0&\mbox{for $j\not=k$}.
\end{array}\right.
\end{equation}
In the definition of $\nabla_H$ we view $D_{H_k}$ as an operator on
$\sD_{H_k}$, and in the definition of $\Pi_H$ the operator
$\Pi_{H_k}$ is the orthogonal projection of $\sU_k$ onto
$\sD_{H_k}$. We can now state the main result of this section.

\begin{thm}\label{th:harmaj}
Let $H\in\eUM^2_\tu{ball}(\asU,\asY)$, and let $\nabla_H$ and
$\Pi_H$ be the diagonal operator matrices given by \eqref{defNaPi}.
Then all operator matrices $W\in\eUM(\asU)$ satisfying
\tu{(\ref{harmaj})} are determined by
\begin{equation}\label{eq:para}
W  = V + \Pi_H^*\nabla_H \left(I +  C \right) \left(I -  C
\right)^{-1} \nabla_H\Pi_H,
\end{equation}
where $V$ in $\eUM({\asU})$ is given by $(\ref{defV})$, and $C$ is
an arbitrary operator matrix in $\eUM^\infty_\tu{ball,0}(\asD_H)$.
Moreover, $W$ and $C$ in \tu{(\ref{harmaj})} determine each other
uniquely.
\end{thm}
\bpr Let $C\in\eUM^\infty_\tu{ball,0}(\asD_H)$. Then $W$ in
(\ref{eq:para}) is equal to
\[
W=V+\Pi_H^*\nabla_H^*K\nabla_H\Pi_H,
\]
where $K$ is the Cayley transform of $C$. Using $\re K$ is
non-negative this yields
\[
\re W=\re V+\Pi_H^*\nabla_H^*\re K \nabla_H\Pi_H\geq \re V=H^*H,
\]
and we have
\[
W_{j,\,j}=V_{j,\,j}+D_{H_j}K_{j,\,j}D_{H_j}=H_j^*H_j+D_{H_j}^2=I_{\sU_j}
\quad (j\in \BZ).
\]
So $W$ satisfies (\ref{harmaj}).

To prove the converse implication,  assume that $W\in\eUM(\asU)$
satisfies (\ref{harmaj}). Since $\re V=H^*H$, the first part of
(\ref{harmaj}) implies that the real part of the operator matrix
$\la=W-V\in\eUM(\asU)$ is non-negative. The second part of
(\ref{harmaj}) gives $\la_{j,\,j}=I_{\sU_j}-H_j^*H_j=D_{H_j}^2$ for
all $j\in\BZ$. The fact that $\re\la$ is non-negative implies that
for any $j< k$ the finite section $\Delta_{j,\, k}(\re \la)$ is a
non-negative operator on $\oplus_{i=j}^k\sU_i$. In particular, the
two by two operator matrix
\[
\mat{cc}{2\la_{j,\,j}&\la_{j,\,k}\\\la^*_{j,\,k}&2\la_{k,\,k}}
=\mat{cc}{2D_{H_j}^2&\la_{j,\,k}\\\la^*_{j,\,k}&2 D_{H_k^2}}
\]
is a non-negative operator on the Hilbert direct sum $\sU_j\oplus
\sU_k$. Recall that an arbitrary operator matrix
\[
 \left[ \begin{array}{cc}
  A_1 & B^* \\
  B &  A_2 \\
\end{array}\right] \mbox{ acting on the Hilbert space direct sum} \left[ \begin{array}{cc}
  \sE_1 \\
  \sE_2 \\
\end{array}\right]
\]
is  a non-negative operator if and only if $A_1$ and $A_2$ are
non-negative and $B = A_2^{1/2} \Phi A_1^{1/2}$, where $\Phi$ is a
contraction from  $\overline{A_1 \sE_1}$ into $\overline{A_2
\sE_2}$. Moreover, in this case, $B$ and $\Phi$ uniquely determine
each other; see Theorem XVI.1.1 in \cite {FF90} for further details.
Thus there exists a unique operator $K_{j,\,k}$ from $\sD_{H_j}$
into $\sD_{H_k}$ such that $\la_{k,\,j}=D_{H_k}K_{k,\,j}D_{H_j}$.
Now set $K_{j,j}=I_{\sD_{H_j}}$ and $K_{j,\,k}=0$ for $j>k$, and put
$K=\mat{c}{K_{j,\,k}}_{j,k\in\BZ}$. Then
\[
K\in\eUM(\asD_H)\quad\mbox{and}\quad
\la=\Pi_H^*\nabla_H^*K\nabla_H\Pi_H.
\]
Since the range of $\nabla_H$ is a dense set in $\asD_H$ and $\la$
is positive real, it follows that $K$ is positive real and that
$\la$ and $K$ determine each other uniquely. Let $C$ be the inverse
Cayley transform of $K$. Then $C\in\eUM^\infty_\tu{ball,0}(\asD_H)$
and $W$ is given by (\ref{eq:para}). It also follows from the above
that $C$ and $W$ determine each other uniquely. \epr

\begin{cor}
 There is only one operator matrix $W\in\eUM(\asU)$
satisfying \tu{(\ref{harmaj})} if and only if for each $k\in\BZ$ the
$k$-th column $H_k$ of $H$ defines an isometry from $\sU_k$ into
$\asY$. In this case $W = V$, where $V$ is given by \eqref{defV}, is
the only operator matrix in $\eUM(\asU)$ satisfying
\tu{(\ref{harmaj})}.
\end{cor}
\bpr The set $\eUM^\infty_\tu{ball,0}(\asD_H)$ consists of just one
element if and only if $\asD_H=\{0\}$, i.e., $\sD_{H_j}=\{0\}$ for
all $j$. The latter condition is equivalent to $H_j$ being an
isometry for each $j\in\BZ$.\epr

\subsection{All solutions to Problem \ref{probl2}} \label{ssecsolprobl2}
We now describe the solution to Problem \ref{probl2}. Fix a
$H\in\eUM^2_\tu{ball}(\asU,\asY)$. Define $N\in\eUM(\asU)$ by
\begin{equation}\label{defN}
N=\mat{c}{N_{j,k}}_{j,k\in\BZ},\quad
N_{j,\,k}=\left\{\begin{array}{cl}
(H_j)^*H_k, & \mbox{for $j\leq k$}, \\
\noalign{\vskip6pt} 0, &\mbox{for $j>k$},
\end{array}\right.
\end{equation}
Here, as before, $H_k$ is the $k$-th column of $H$.

\begin{thm}\label{th:Fexists}
Let $H\in\eUM^2_\textup{ball}(\asU,\asY)$. Define $N\in\eUM(\asU)$
by \eqref{defN}, and let the matrices $\nabla_H\in \eUM(\asD_H)$ and
$\Pi_H\in \eUM(\asD_H, \asU)$ be the diagonal operator matrices
given by \eqref{defNaPi}. For each
$C\in\eUM^\infty_\tu{ball,0}(\asD_H)$ put
\begin{equation}\label{CtoF}
F=N+\Pi_H^*\nabla_H(I_{\asD_H}-C)^{-1}\nabla_H\Pi_H.
\end{equation}
Then $F\in\eUM(\asU)$, and  the map $C\mapsto F$ defined by
\eqref{CtoF} is a one-to-one map from
$\eUM^\infty_\tu{ball,0}(\asD_H)$ onto the set of all
$F\in\eUM(\asU)$ that satisfy \eqref{major}. In particular,
$F\in\eUM(\asU)$ satisfying \eqref{major} exist. Finally, there
exists a unique $F\in\eUM(\asU)$ satisfying \eqref{major} if and
only if for each $k\in\BZ$ the $k$-th column $H_k$ of $H$ defines an
isometry from $\sU_k$ into $\asY$. In this case $F = N$ is the only
$F\in\eUM(\asU)$ satisfying \eqref{major}.
\end{thm}
\bpr Assume $W$ and $F$ belong to $\eUM(\asU)$ and determine each
other uniquely via
\begin{equation}\label{WtoFtoW}
W=2F-I_\asU\ands F=\half(W+I).
\end{equation}
Then $W_{j,\,j}=I_{\sU_j}$ if and only if $F_{j,\,j}=I_{\sU_j}$ for
each $j\in\BZ$. Moreover, $2\re F=\re W+I$. Hence $F$ satisfies
(\ref{HFconditions}) (where $F+F^*=2\re F$) if and only if $W$
satisfies (\ref{harmaj}).

To complete the proof it remains to show that the map $C\mapsto W$
given by (\ref{eq:para}) composed with the map $W\mapsto F$ in
(\ref{WtoFtoW}) gives the map $C\mapsto F$ in (\ref{CtoF}). Let
$C\in\eUM^\infty_\tu{ball,0}(\asD_H)$, define $W$ by (\ref{eq:para})
and $F$ by (\ref{WtoFtoW}). Notice that $V$ in (\ref{defV}) and $N$
in (\ref{defN}) are related via $V=2N+\Pi_H^*\nabla_H^2\Pi_H-I.$
Thus
\begin{eqnarray*}
W&=&V+\Pi_H^*\nabla_H(I+C)(I-C)^{-1}\nabla_H\Pi_H\\
&=&2N +\Pi_H^*\nabla_H^2\Pi_H+\Pi_H^*\nabla_H(I+C)(I-C)^{-1}\nabla_H\Pi_H-I\\
&=&2N+\Pi_H^*\nabla_H(I+C+I-C)(I-C)^{-1}\nabla_H\Pi_H-I\\
&=&2N+2\Pi_H^*\nabla_H(I-C)^{-1}\nabla_H\Pi_H-I.
\end{eqnarray*}
Hence
\begin{eqnarray*}
F&=&\half(W+I)=\half (2N+2\Pi_H^*\nabla_H(I-C)^{-1}\nabla_H\Pi_H)\\
&=&N+\Pi_H^*\nabla_H(I-C)^{-1}\nabla_H\Pi_H.
\end{eqnarray*} So $F$ is given by (\ref{CtoF}).\epr

\bigskip\noindent\textbf{A state space example.}
Consider the state space system $\{A,B,E,D\}$, where $A$ is an
operator on $\sX$ whose spectrum is contained in the open unit disc.
The input space $\asU = \oplus_{j=1}^{\ n} \sU_j$ and the output
space $\asY =\oplus_{j=1}^{\ n} \sY_j$. Furthermore,  $B$ is an
operator mapping $\asU$ into  $\sX$ and $E$ is an operator mapping
$\sX$ into $\asY$,  while $D\in {\eUM}(\asU, \asY)$ is a finite
upper triangular operator matrix mapping  $\asU= \oplus_{j=1}^{\ n}
\sU_j $ into $\asY =\oplus_{j=1}^{\ n} \sY_j$.

Now, let  $H$ be the operator matrix (consisting of $N$ doubly
infinite columns)
 of which the $k$-th column
$H_k=\textup{col}\{H_{j,\,k}\}_{j\in \BZ}$, $k=1, \ldots,n$, is
given by
\[
H_{j,\,k}=\left\{\begin{array}{cl} EA^{-j-1}B_k&\mbox{when $j<0$},\\
D_{j,\,k}&\mbox{when $j=1, \dots,n$},\\
0&\mbox{when $j>n$}.
\end{array}\right.
\]
Here $D_{j,\,k}$ is the $(j,\,k)$-th entry of the $n\ts n$ operator
matrix $D$, and $B_k$ is the restriction of $B$ to the $k$-th
component $\sU_k$ of $\asU = \oplus_{j=1}^{\ n} \sU_j$. For this $H$
we consider the finite operator matrix version of Problem
\ref{probl2}, that is, we seek all $F\in\eUM(\asU)$ satisfying
\begin{equation}
\label{FHH*2} F + F^* \geq H^*H +I_\asU \ands F_{j,j} =
I_{\sU_j}\quad(j=1,\dots,n).
\end{equation}
By setting $\sU_j=\{0\}$ for $j<0$ and $j>n$, we can identify $\asU$
and $\oplus_{j\in \BZ}\sU_j$, and we may view $H$ as an  operator
matrix in $\eM(\asU, \asE)$, where $\asE=\oplus_{j\in \BZ}\sE_j$
with
\[
\sE_j=\left\{\begin{array}
{ll}\asY&\mbox{when $j<0$ or $j>n$},\\
\sY_j &\mbox{when $j=1, \dots,n$}.
\end{array}\right.
\]
The fact that $D$ is upper triangular implies that $H\in \eUM(\asU,
\asE)$. Hence the results obtained above apply.

Since $A$ has its spectrum in the open unit disc, we know that the
Lyapunov equation $P = A^* P A + E^*E$ has a unique solution which
is given by
\[
P = \sum_{j=0}^\infty A^{*j}E^*E A^j.
\]
Using this,  we obtain $H^*H = D^*D + B^* P B$. It follows that for
each $k=1, \ldots,n$ the $k$-th column $H_k$ of $H$ induces a
contraction from $\asU$ into $\asE$  if and only if $P_{\sU_k} (D^*D
+ B^* P B)|_{\sU_k}$ is a contraction, where $P_{\sU_k}$ is the
orthogonal projection onto $\sU_k$. In other words, $H$ is in
$\eUM^2_\tu{ball}({\asU},{\asE})$  if and only if the block diagonal
entries of $D^*D + B^* P B$ are contractions.

Now assume that $H$ is in $\eUM^2_\tu{ball}({\asU},{\asE})$.  Let
$N$ be the upper triangular part of $D^*D + B^* P B$; see
\eqref{defN}. In this setting,  $D_{H_k}$ equals the positive square
root of the operator $I -  P_{\sU_k} (D^*D + B^* P B)|_{\sU_k}$, and
$\nabla_H$ is the diagonal operator  formed by $\{D_{H_k}\}_{k=1}^n$
on $\oplus_{k=1}^{\ n} \sD_{H_k}$; see \eqref{defNaPi}. Then the set
of all operators   $F\in\eUM(\asU)$ satisfying \eqref{FHH*2} is
determined by \eqref{CtoF}. Finally, $H=Z^{(1)}(I_\asU-
Z^{(2)})^{-1}$, where $Z^{(1)} = H F^{-1}$ and $Z^{(2)} = I -
F^{-1}$. Moreover, the operator matrix $\begin{bmatrix}
Z^{(1)*}&Z^{(2)*}
\end{bmatrix}$
induces a  contraction.

\section{The full version of Theorem \ref{thmmain2} and its proof}\label{sec:proof3}\setcounter{equation}{0}
The following theorem is the full version of Theorem \ref{thmmain2}.

\begin{thm}
\label{thmmain2full} Let $H$ be a solution to \textup{Problem}
\ref{mainprobl}, and let $\eC_{H,\, \o}$ be the set of all operator
matrices $C$ in $\eUM^\infty_\tu{ball,0}(\asD_H)$ defined by
\eqref{paramCH}. Fix $C\in\eC_{H,\, \o}$, and  put
\begin{equation}\label{defFC}
F=N+\Pi_H^*\nabla_H(I_{\asD_H}-C)^{-1}\nabla_H\Pi_H.
\end{equation}
Here $N\in\eUM(\asU)$ is defined by \eqref{defN}, and $\nabla_H\in
\eUM(\asD_H)$ and $\Pi_H\in \eUM(\asD_H, \asU)$ are the diagonal
operator matrices given by \eqref{defNaPi}. Then  $F\in\eUM(\asU)$,
$F$ satisfies \eqref{major}, and
\begin{equation}\label{intpolcondF2}
F_{j,\,k}|_{\sF_k}=F_{j,\,k-1}\o_{k,2}\quad(j,k\in\BZ,j<k).
\end{equation}
Put
\begin{equation}\label{Z1Z2b} Z^{(1)}= HF^{-1} \ands Z^{(2)}=
I_{{\asU}}- F^{-1}.
\end{equation}
Then the pair of operator matrices $(Z^{(1)}, Z^{(2)})$ satisfies
\eqref{Z12}, \eqref{SchurClass}, \eqref{condZ12}, and $H$ is given
by \eqref{HZ12}. Furthermore, the map $C\mapsto (Z^{(1)}, Z^{(2)})$
is a one-to one map   from the set $\eC_{H,\, \o}$ onto the set of
all  pairs $(Z^{(1)}, Z^{(2)})$ satisfying \eqref{Z12},
\eqref{SchurClass}, \eqref{condZ12}, and such that $H$ is given by
\eqref{HZ12}. In particular, there exists a unique pair of operator
matrices $(Z^{(1)}, Z^{(2)})$ satisfying \eqref{Z12},
\eqref{SchurClass} and \eqref{condZ12} such that $H$ is given by
\eqref{HZ12} if and only if one of the following three conditions is
satisfied:
\begin{itemize}
\item[\tu{(1)}] $\sF_{H_k}=\sD_{H_k}$ for each $k\in\BZ$;

\item[\tu{(2)}] $\o_{H_k}$ is a co-isometry for each $k\in\BZ$;

\item[\tu{(3)}] there exists an integer $k\in\BZ$ such that $\sF_{H_j}=\sD_{H_j}$ for each $j>k$  and
the operator $\o_{H_j}$ is a co-isometry for each $j\leq k$.
\end{itemize}
\end{thm}

Let $H$ be a solution to Problem \ref{mainprobl}. Recall that, in
particular, this implies that $H$ belongs to
$\eUM^2_\textup{ball}(\asU,\asY)$. Hence Theorem \ref{th:Fexists}
applies to $H$. It will be convenient first to prove the following
proposition.

\begin{prop}
\label{pr:paramFC}Let $H$ be a solution to Problem \ref{mainprobl}.
Fix $C$ in $\eUM^\infty_\tu{ball,0}(\asD_H)$, and let
$F=[F_{j,\,k}]_{j,k\in\BZ}$ in $\eUM(\asU)$ by defined by
\eqref{defFC}. Then $F$ satisfies \eqref{intpolcondF2} if and only
if $C$ belongs to the set $\eC_{H,\, \o}$ defined by
\eqref{paramCH}.
\end{prop}
\bpr Let $H$ be a solution to Problem \ref{mainprobl}. We first show
that
\begin{equation}\label{IP4}
C\in \eC_{H,\, \o}\quad \Longleftrightarrow\quad
C\nabla_H\Pi_HE=\nabla_H\Pi_H\om^{(2)}.
\end{equation}
Here $E$ and $\om^{(2)}$ are the operator matrices in $\eUM(\asF,
\asU)$ defined by \eqref{defE} and \eqref{defOm12}. To prove
\eqref{IP4} we use \eqref{defomHk} and \eqref{KH}. {}From these
formulas it follows that the operator matrix
$C=[C_{j,\,k}]_{j,k\in\BZ}$ in $\eUM^\infty_\tu{ball,0}(\asD_H)$
belongs to $\eC_{H,\, \o}$ if and only if for each $j<k$ in $\BZ$
and each $f_k\in\sF_k$ we have
\begin{equation}
\label{entriesC} C_{j,\,k}D_{H_k}f_k=\left\{\begin{array}{cl}
D_{H_{k-1}}\o_{k,\,2}f_k & \mbox{for $j=k-1$}, \\
\noalign{\vskip6pt} 0, &\mbox{for $j\not=k-1$},
\end{array}\right.\quad f_k\in\sF_k\quad (j<k\in\BZ).
\end{equation}
In the language of operator matrices \eqref{entriesC} is equivalent
to the right hand side of \eqref{IP4}. Thus our claim follows.

By assumption $H$ belongs to $\eUM^2_\textup{ball}(\asU,\asY)$.
Hence  Theorem \ref{th:Fexists} applies to $H$. In particular, since
\eqref{CtoF}, and \eqref{defFC} are the same identities, $F$ is
well-defined and belongs to $\eUM(\asU)$. {}From \eqref{CtoF} it
follows that $F$ is also given by the following formula:
\begin{equation}
\label{CtoF2} F=I_{\asU}+\tilde
N+\Pi_H^*\nabla_H(I-C)^{-1}C\nabla_H\Pi_H,
\end{equation}
where $\tilde N=[\tilde N_{i,\,j}]\in\eUM(\asU)$ is the strictly
upper triangular operator matrix given by
\begin{equation}\label{deftilN}
\tilde N_{i,\,j}=\left\{\begin{array}{cl}
(H_i)^*H_j, & \mbox{for $i<j$}, \\
\noalign{\vskip6pt} 0, &\mbox{for $i\geq j$}.
\end{array}\right.
\end{equation}
We claim that
\begin{equation}
\label{NtildeE} \tilde{N}E=N\om^{(2)}.
\end{equation}
Note that both $\tilde{N}E$ and $N\om^{(2)}$ are strictly upper
triangular operator matrices in $\eUM(\asF, \asU)$. Furthermore, for
$j<k$ the $(j,k)$-th entry of $\tilde{N}E$ is equal to
$(H_j)^*H_k\t_k$, where $\t_k$ is the canonical embedding of $\sF_k$
into $\sU_k$. Now observe, using the second part of
\eqref{intpolcond1}, that  for $j<k$ we have
\begin{eqnarray*}
(H_j)^*H_{k}\t_k f_k&=&\sum_{\nu=-\iy}^j
H_{\nu,\,j}^*H_{\nu,\,k}f_k=\sum_{\nu=-\iy}^j
H_{\nu,\,j}^*H_{\nu,\,k-1}\o_{k,\,2}f_k\\
&=&(H_j)^*H_{k-1}\o_{k,\,2}f_k\quad (f_k\in \sF_k).
\end{eqnarray*}
But $(H_j)^*H_{k-1}\o_{k,\,2}$ is precisely the $(j,k)$-th entry of
$N\om^{(2)}$. Thus \eqref{NtildeE} is proved.

Next we use the two representations for $F$ given by \eqref{defFC}
and \eqref{CtoF2}. We multiply \eqref{defFC} and \eqref{CtoF2} from
the right by $\om^{(2)}$ and $E$, respectively, and subtract the
resulting identities. Then, using  \eqref{NtildeE}, we obtain
\[
F\om^{(2)}-FE+E=\Pi_H^*\nabla_H(I_{\asD_H}-C)^{-1}(\nabla_H\Pi_H\om^{(2)}-C\nabla_H\Pi_HE).
\]
Notice that $\nabla_H(I_{\asD_H}-C)^{-1}$ acts as a one-to-one
linear transformations on the linear space $\asD_H^+$ (see the final
paragraph of Subsection \ref{ssecinvUM} for the definition of this
space).  Since $\Pi_H^*$ is also one-to-one on vectors from
$\asD_H^+$, it follows that
\begin{equation}\label{finequiv}
F\om^{(2)}-FE+E=0 \quad \Longleftrightarrow\quad
\nabla_H\Pi_H\om^{(2)}-C\nabla_H\Pi_HE=0.
\end{equation}
By comparing the left hand side of \eqref{finequiv} with \eqref{IP3}
and the right hand side of \eqref{finequiv} with \eqref{IP4}, we see
that $F$ satisfies \eqref{intpolcondF2} if and only if $C$ belongs
to $\eC_{H,\, \o}$. \epr

\medskip\noindent\textbf{Proof of Theorem \ref{thmmain2full}.}
Let $H$ be a solution to \textup{Problem} \ref{mainprobl}. Let $F$
be given by \eqref{defFC} with $C$ from $\eC_{H,\, \o}$. By
assumption $H$ satisfies \eqref{intpolcond1}. Proposition
\ref{pr:paramFC} tells us that $F$ satisfies \eqref{intpolcondF}.
Thus we can apply Proposition \ref{pr:UM2intpol} to show that the
pair $(Z^{(1)}, Z^{(2)})$ satisfies the interpolation condition
\eqref{condZ12}. To see that the map $C\mapsto (Z^{(1)}, Z^{(2)})$
is a one-to one map from the set $\eC_{H,\, \o}$ onto the set of all
pairs $(Z^{(1)}, Z^{(2)})$ satisfying \eqref{Z12},
\eqref{SchurClass}, \eqref{condZ12} and such that $H$ is given by
\eqref{HZ12}  it  remains to apply Theorems \ref{th:UM2} and
\ref{th:Fexists}.

In order to prove  the  claim in the final part  of Theorem
\ref{thmmain2full} note that there exists a unique pair of operator
matrices $(Z^{(1)}, Z^{(2)})$ satisfying \eqref{Z12},
\eqref{SchurClass} and \eqref{condZ12} such that $H$ is given by
\eqref{HZ12} if and only if the set $\eC_{H,\, \o}$ is a singleton.

Let $C=\mat{c}{C_{i,\, j}}_{i,j\in\BZ}\in\eC_{H,\, \o}$. Fix a
$k\in\BZ$. Then observe that $\sF_{H_k}=\sD_{H_k}$ implies that the
$k$-th column of $C$ is completely determined by
\begin{equation*}
C_{i,\,k}=\left\{\begin{array}{cl}
\o_{H_k},&\mbox{if $i=k-1$},\\
0,&\mbox{if $i\not=k-1$}.
\end{array}\right.
\end{equation*}
Moreover, if $\o_{H_k}$ is a co-isometry, then we can use  Corollary
XXVII.5.3 in \cite{GGK93} to show that the $(k-1)$-th row of $C$ is
completely determined by
\begin{equation*}
C_{k-1,\,j}=\left\{\begin{array}{cl}
\o_{H_k}\Pi_{\sF_{H_k}},&\mbox{if $j=k$},\\
0,&\mbox{if $j\not=k$}.
\end{array}\right.
\end{equation*}
Here $\pi_{\sF_{H_k}}$ denotes the orthogonal projection from
$\sD_{H_k}$ onto $\sF_{H_k}$. From these two observations and the
fact that $C$ is strictly upper triangular it follows that the
conditions (1), (2) and (3) are each sufficient for $\eC_{H,\, \o}$
to be a singleton.

To see that these conditions are also necessary, assume that non of
the conditions (1), (2) or (3) is satisfied, i.e., assume there
exists a $k\in\BZ$ with $\sF_{H_k}\not=\sD_{H_k}$ and such that
$\o_{H_{k-1}}$ in not a co-isometry. In that case, set
$\sG_{H_k}=\sD_{H_k}\ominus\sF_{H_k}$ and let $D_{\o_{H_{k-1}}}$ and
$\sD_{\o_{H_{k-1}}}$ denote the defect operator and defect space of
$\o_{H_{k-1}}$, respectively. Then both $\sG_{H_k}$ and
$\sD_{\o_{H_{k-1}}}$ are not equal to $\{0\}$, and thus there exists
a non-zero contraction $N$ from $\sG_{H_k}$ into
$\sD_{\o_{H_{k-1}}}$. Now define $C=\mat{c}{C_{i,\,
j}}_{i,j\in\BZ}\in\eUM^\infty_\tu{ball,0}(\asD_{H})$ by setting
$C_{i,\,j}=0$ in case $i\not=j-1$ and $(i,\,j)\not=(k-2,k)$,
$C_{j-1,\,j}=\o_{H_j}\Pi_{\sF_{H_j}}$ for each $j\in\BZ$ and
$C_{k-1,\,k}=D_{\o_{k-1}^*}N\Pi_{\sG_{H_k}}$. One easily sees that
$C$ is in $\eC_{H,\, \o}$. Moreover, this is not the only element of
$\eC_{H,\, \o}$, because $\eC_{H,\, \o}$ always contains the
operator matrix in $\eUM^\infty_\tu{ball,0}(\asD_H)$ that has zeros
in all entries accept for the first upper diagonal on which
$\o_{H_k}\Pi_{\sF_{H_k}}$ is the entry in the $(k-1,k)$-th
position.\epr

\medskip
The arguments used to prove the claim in the final part of Theorem
\ref{thmmain2full} can also be used to derive the following
proposition.

\begin{prop}\label{P:UniqueSol}
There exists a unique pair of operator matrices $(Z^{(1)}, Z^{(2)})$
satisfying \eqref{Z12}, \eqref{SchurClass} and \eqref{condZ12} if
and only if one of the following three conditions is satisfied:
\begin{itemize}
\item[\tu{(1)}] $\sF_{k}=\sU_{k}$ for each $k\in\BZ$;

\item[\tu{(2)}] $\o_{k}$ is a co-isometry for each $k\in\BZ$;

\item[\tu{(3)}] there exists a $k\in\BZ$ such that $\sF_{j}=\sU_{j}$ for each $j>k$  and
the operator $\o_{j}$ is a co-isometry for each $j\leq k$.
\end{itemize}
In particular, if one of the conditions \tu{(1)}, \tu{(2)} or
\tu{(3)} is satisfied, then there exists a unique solution to
Problem \ref{mainprobl}.
\end{prop}
\bpr Let $(Z^{(1)},Z^{(2)})$ be a pair of operator matrices as in
(\ref{Z12}). Set $\asZ=\oplus_{i\in\BZ} (\sY_{i+1}\oplus\sU_i)$,
and define
\[
\widetilde{Z}=\mat{c}{\widetilde{Z}_{i,\,j}}_{i,j\in\BZ}\in\eUM^\infty_0(\asU,\asZ),\
\text{where}\ \widetilde{Z}_{i,\,j}=\mat{c}{Z^{(1)}_{i+1,\,j}\\
Z^{(2)}_{i,\,j}}:\sU_j\to\mat{c}{\sY_{i+1}\\\sU_i}.
\]
Observe that (\ref{SchurClass}) is equivalent to
$\widetilde{Z}\in\eUM^\infty_\tu{ball,0}(\asU,\asZ)$, while
(\ref{condZ12}) corresponds to
\begin{equation}\label{tilZcon}
\widetilde{Z}_{j,\,k}|_{\sF_k}=\left\{\begin{array}{cl}
\o_{k},&\mbox{if $j=k-1$},\\
0,&\mbox{if $j\not =k-1$}.
\end{array}\right.
\end{equation}
Thus the pair $(Z^{(1)},Z^{(2)})$ satisfies (\ref{SchurClass}) and
(\ref{condZ12}) if and only if $\widetilde{Z}$ is an element of
\begin{eqnarray*}
&&\eC_{\o}=\{\widetilde{Z}\in \eUM^\infty_\tu{ball,0}(\asU,\asZ)
\mid
\mbox{the $(i,j)$-th entry $\widetilde{Z}_{i,\,j}$}\nonumber\\
\noalign{\vskip6pt} &&\hspace{3.5cm}\mbox{of $\widetilde{Z}$
satisfies \eqref{tilZcon} for each $i,j\in\BZ$}\}.
\end{eqnarray*}
The first statement now follows by  translating  the arguments in
the proof of the last part of Theorem \ref{thmmain2full}  to the
present setting. The last statement of Proposition \ref{P:UniqueSol}
follows immediately from the first part. \epr

\medskip
The conditions listed in the above proposition are sufficient, but
in general not necessary conditions for the existence of a unique
solution. This is already the case in the time-invariant case; see
\cite{LT06,tH2}. The problem to give necessary and sufficient
conditions for the existence of a unique solution to Problem
\ref{mainprobl} remains open.

\section{An example involving finite operator matrices and
time-variant relaxed commutant
lifting}\label{sec:examples}\setcounter{equation}{0} In this section
we present some examples of how  our results can be applied. In each
case the contractions $\o_k$, $k\in \BZ$, are not given beforehand
but are constructed from the given data. The first subsection deals
with a $4\ts 4$ operator matrix problem. In the second subsection we
introduce a time-variant analogue of the relaxed commutant lifting
problem, and show how this time-variant problem can be solved by
using Theorem \ref{thmmain1}.

\subsection{An example involving finite operator matrices}
When in Problem \ref{mainprobl} all spaces $\sU_k$ and $\sY_j$ are
set to zero, with the exception of a finite numbers of $k$'s and
$j$'s, finite operator matrix problems appear. We illustrate this
with an example.

Consider the problem of finding all $4\ts 4$ operator matrices
\begin{equation}
\label{opmat}A=\begin{bmatrix} A_{1,1}&A_{1,2}&A_{1,3}&A_{1,4}\\
A_{2,1}&A_{2,2}&A_{2,3}&A_{2,4}\\
A_{3,1}&A_{3,2}&A_{3,3}&A_{3,4}\\
A_{4,1}&A_{4,2}&A_{4,3}&A_{4,4}
\end{bmatrix}:\begin{bmatrix}
\sX_1\\
\sX_2\\
\sX_3\\
\sX_4
\end{bmatrix}\to \begin{bmatrix}
\sY_1\\
\sY_2\\
\sY_3\\
\sY_4
\end{bmatrix}
\end{equation}
such that \begin{equation}\label{condopmat}
\|\begin{bmatrix} A_{1,k}&A_{1,k+1}\\
A_{2,k}&A_{2,k+1}\\
A_{3,k}&A_{3,k+1}\\
A_{4,k}&A_{4,k+1}
\end{bmatrix}\|\leq 1, \quad k=1,2,3.
\end{equation}
The problem is always solvable (the zero matrix is a solution).
Moreover all solutions can be obtained by repeatedly applying
Parrott's lemma \cite{P78} (cf., Section IV.1 in \cite{FF90} or
Section XXVII.5 in \cite{GGK93}). Here we show that a more direct
description of the  set of all solutions  can be obtained from
Theorem \ref{thmmain1}. To do this we set $\sY_j=\{0\}$ for
$j\not=1,2,3,4$ and put
\begin{eqnarray*}
&& \sU_k=\{0\}, \quad \sF_k=\{0\}\quad (k\not =4,5,6), \\
&&\sU_4=\sX_1\oplus \sX_2, \quad \sF_4=\{0\}, \\
&&\sU_5=\sX_2\oplus \sX_3, \quad \sF_5=\sX_2\oplus \{0\},
\quad \s_{5}(x_2\oplus 0)=0\oplus x_2\quad (x_2\in \sX_2),\\
&&\sU_6=\sX_3\oplus \sX_4, \quad \sF_6=\sX_3\oplus \{0\}, \quad
\s_{6}(x_3\oplus 0)=0\oplus x_3\quad (x_2\in\sX_3).
\end{eqnarray*}
To formulate this as an upper triangular operator matrix problem,
consider
\begin{eqnarray*}
H_{1,\,4}=\begin{bmatrix}A_{1,1}&A_{1,2}\end{bmatrix}, \quad
H_{1,\,5}=\begin{bmatrix}A_{1,2}&A_{1,3}\end{bmatrix}, \quad
H_{1,\,6}=\begin{bmatrix}A_{1,3}&A_{1,4}\end{bmatrix}\\
H_{2,\,4}=\begin{bmatrix}A_{2,1}&A_{1,2}\end{bmatrix}, \quad
H_{2,\,5}=\begin{bmatrix}A_{2,2}&A_{2,3}\end{bmatrix}, \quad
H_{2,\,6}=\begin{bmatrix}A_{2,3}&A_{2,4}\end{bmatrix}\\
H_{3,\,4}=\begin{bmatrix}A_{3,1}&A_{3,2}\end{bmatrix}, \quad
H_{3,\,5}=\begin{bmatrix}A_{3,2}&A_{3,3}\end{bmatrix}, \quad
H_{3,\,6}=\begin{bmatrix}A_{3,3}&A_{3,4}\end{bmatrix}\\
H_{4,\,4}=\begin{bmatrix}A_{1,1}&A_{1,2}\end{bmatrix}, \quad
H_{4,\,5}=\begin{bmatrix}A_{1,2}&A_{1,3}\end{bmatrix}, \quad
H_{4,\,6}=\begin{bmatrix}A_{1,3}&A_{1,4}\end{bmatrix}\\
\end{eqnarray*}
Then \eqref{opmat} and \eqref{condopmat} are satisfied if and only
if
\begin{eqnarray*}
&&\hspace{2cm}\|\begin{bmatrix}
H_{1,\,k}\\
H_{2,\,k}\\
H_{3,\,k}\\
H_{4,\,k}
\end{bmatrix}\|\leq 1 \quad (k=4,5,6), \\
&&\begin{bmatrix}
H_{1,\,5}\\
H_{2,\,5}\\
H_{3,\,5}\\
H_{4,\,5}
\end{bmatrix}|_{\sF_5}=\begin{bmatrix}
H_{1,\,4}\\
H_{2,\,4}\\
H_{3,\,4}\\
H_{4,\,4}
\end{bmatrix}\s_5, \ands \begin{bmatrix}
H_{1,\,6}\\
H_{2,\,6}\\
H_{3,\,6}\\
H_{4,\,6}
\end{bmatrix}|_{\sF_5}=\begin{bmatrix}
H_{1,\,5}\\
H_{2,\,5}\\
H_{3,\,5}\\
H_{4,\,5}
\end{bmatrix}\s_6.
\end{eqnarray*}
Now consider our main problem with $\sY_j$, $\sU_k=\{0\}$ and
$\sF_k$ as above. Furthermore, put
\[
\o_{5,\,1}=0, \quad \o_{6,\,1}=0, \ands \o_{5,\,2}=\s_5\quad
\o_{6,\,2}=\s_6.
\]
Since $\sF_k=\{0\}$ for $k\not =5,6$, we don't have to consider the
operators $\o_{k,\,1}$ and $\o_{k,\,1}$ for $k\not =5,6$. Note that
\[
\begin{bmatrix}
\o_{5,\,1}\\ \o_{6,\,1}
\end{bmatrix}\ands\begin{bmatrix}
\o_{5,\,2}\\ \o_{6,\,2}
\end{bmatrix}\quad\mbox{are both contractions.}
\]
It is now clear how the problem regarding the operator matrix $A$ in
\eqref{opmat} can be solved via Theorem \ref{thmmain1}.


\subsection{A time-variant relaxed commutant lifting problem}

In this subsection we present a time-varying version of the relaxed
commutant lifting problem from \cite{FFK02} and explain the
connection with Problem \ref{mainprobl}. We plan to come back to
this time-variant relaxed commutant lifting problem in more detail
in a future publication, where we will also discuss the relation
with the time-invariant version, the three chain completion problem
\cite{FFGK97a,FFGK97b} and its weighted versions \cite{BFF01}.

A data set for the time-variant relaxed commutant lifting problem is
a set $\la=\{A_n,T'_n,U'_n,R_n,Q_n\mid n\in\BZ\}$ consisting of
Hilbert space operators with for each $k\in\BZ$
\[
T_k':\sH_{k-1}'\to\sH_k',\quad A_k:\sH_k\to\sH_k',\quad
R_k:\sH_{0,k}\to\sH_k,\quad Q_k:\sH_{0,k-1}\to\sH_k,
\]
and such that $T_k'$ and $A_k$ are contractions and
\[
T_{k}'A_{k-1}R_{k-1}=A_{k}Q_{k},\quad R_{k-1}^*R_{k-1}\leq
Q_{k}^*Q_{k}.
\]
The operator $U_k'$ is completely determined by the set of operators
$\{T_n'\mid n\in\BZ\}$ and can be seen as a time-varying analogue of
the Sz.-Nagy-Sch\"affer isometric lifting of $T_k'$; cf.,
\cite{C90}. To define $U_k'$ we set
\[
\asD'_n=\oplus_{i\leq n}\sD_{T'_{n}} \ands
\sK'_n=\sH_n\oplus\asD'_n\quad(n\in \BZ).
\]
Then $U_k'$ is the isometric operator mapping $\sK_{k-1}'$ into
$\sK_k'$ given by
\begin{equation}\label{U'k}
U_k'=\mat{cc}{T_k'&0\\E_{\sD_{T_n'}}D_{T_n'}&E_{\asD'_{k-1}}}
:\mat{c}{\sH_{k-1}'\\\asD'_{k-1}}\to\mat{c}{\sH_{k}'\\\asD'_{k}}.
\end{equation}
Here $E_{\sD_{T'_k}}$ and $E_{\asD'_{k-1}}$ are the canonical
embeddings of $\sD_{T'_k}$ and $\asD'_{k-1}$ into $\asD'_{k}$,
respectively. We then consider the following problem.

\begin{probl}\label{TV-RCLProblem}
Given the data set $\la=\{A_n,T'_n,U'_n,R_n,Q_n\mid n\in\BZ\}$,
describe the sets of operators $\{B_n\mid n\in\BZ\}$ with the
property that for each $k\in\BZ$ the operator $B_k$ is a contraction
from $\sH_k$ into $\sK_k'$ satisfying
\begin{equation}\label{reqB}
\Pi_{\sH'_k}B_k=A_k\ands U'_{k}B_{k-1}R_{k-1}=B_{k}Q_{k}.
\end{equation}
Here $\Pi_{\sH_k'}$ is the orthogonal projection from $\sK_k'$ onto
$\sH_k'$.
\end{probl}

After some translation and reduction steps it follows that the
special case of Problem \ref{TV-RCLProblem} with $\sH_{0,k}=\sH_k$
and $R_k=I_{\sH_k}$ for each $k\in\BZ$ is just the nonstationary
commutant lifting problem considered in \cite{C90} (see also Section
3.5 in \cite{C96}).

With the data set $\la$ we associate a set of contractions
$\{\o_n\mid n\in\BZ\}$ of the form (\ref{defomk}). For each
$k\in\BZ$, the contraction $\o_k$ is defined on the subspace
$\sF_k=\ov{D_{A_k}Q_{k}\sH_{0,k-1}}$ of $\sD_{A_k}$ and is given by
\[
\o_k:\sF_k\to\mat{c}{\sD_{T'_k}\\\sD_{A_{k-1}}},\quad \o_k
D_{A_k}Q_{k}=\mat{c}{D_{T'_k}A_{k-1}R_{k-1}\\D_{A_{k-1}}R_{k-1}}.
\]
We refer to $\o_k$ as the {\em $k$-th underlying contraction of the
data set $\la$}. To see that $\o_k$ is in fact a contraction observe
that for each $h\in \sH_{0,k-1}$ we have
\begin{eqnarray*}
\|D_{A_k}Q_{k} h\|^2&=&\|Q_kh\|^2-\|A_kQ_kh\|^2
\leq\|R_{k-1}h\|^2-\|T_k'A_{k-1}R_{k-1}h\|^2\\
&=&\|R_{k-1}h\|^2-\|A_{k-1}R_{k-1}h\|^2+\|A_{k-1}R_{k-1}h\|^2-\|T_k'A_{k-1}R_{k-1}h\|^2\\
&=&\|D_{A_{k-1}}R_{k-1}h\|^2+\|D_{T_k'}A_{k-1}R_{k-1}h\|^2.
\end{eqnarray*}
It follows directly from this computation that $\o_k$ is
contractive, and moreover, that $\o_k$ is an isometry if and only if
$R_{k-1}^*R_{k-1}= Q_{k}^*Q_{k}$.

{}From a variation on Parrott's lemma \cite{P78} (cf., Section IV.1
in \cite{FF90} or Section XXVII.5 in \cite{GGK93}) we obtain that an
operator $B_k$  from $\sH_k$ into $\sK_k'$ is a contraction that
satisfies the first requirement of (\ref{reqB}) if and only $B_k$ is
of the form
\begin{equation}\label{GaB}
B_k=\mat{c}{A_k\\\ga_k D_{A_k}},\quad\text{with
}\ga_k:\sD_{A_k}\to\asD'_k\text{ a contraction}.
\end{equation}
Moreover, the contraction $\ga_k$ is uniquely determined by $B_k$.
In this case, writing out $U_{k}'B_{k-1}R_{k-1}$ and $B_{k}Q_k$, it
follows that the second requirement in (\ref{reqB}) holds if and
only if $\ga_k$ and $\ga_{k-1}$ satisfy
\begin{equation}\label{sol1}
\ga_k|_{\sF_k}=E_{\sD_{T'_k}}\o_{k,1}+E_{\asD'_{k-1}}\ga_{k-1}\o_{k,2}.
\end{equation}
Here $\o_{k,1}$ is the first component of $\o_k$ mapping $\sF_k$
into $\sD_{T_k'}$ and $\o_{k,2}$ is the second component of $\o_k$
mapping $\sF_k$ into $\sD_{A_{k-1}}$. Thus, alternatively, we seek
the sets of operators $\{\ga_n\mid n\in\BZ\}$ such that for each
$k\in\BZ$ the operator $\ga_{k}$ is a contraction from $\sD_{A_{k}}$
into $\asD'_k$ and (\ref{sol1}) is satisfied.

Now set $\asD=\oplus_{k\in\BZ}\sD_{A_k}$ and
$\asD'=\oplus_{k\in\BZ}\sD_{T'_n}$. With a set of contractions
$\{B_n:\sH_n'\to\sK_n'\mid n\in\BZ\}$ satisfying the first condition
of (\ref{reqB}) we associate an operator matrix
$H\in\eUM^2_\tu{ball}(\asD,\asD')$ in the  following way.

\begin{proc}\label{pc:BtoH}
Let $\{B_n:\sH_n'\to\sK_n'\mid n\in\BZ\}$ be a set of contractions
such that the first condition of \tu{(\ref{reqB})} holds for each
$k\in\BZ$, and let $\ga_k$ be the contraction from $\sD_{A_k}$ into
$\asD_k'$ determined by \tu{(\ref{GaB})}. For each $k\in\BZ$ we view
$\asD'_k$ as a subspace of $\asD'$ and write $\Pi_{\asD'_k}$ for the
orthogonal projection from $\asD'$ onto $\asD'_k$. Define $H_k$ to
be the contraction from $\sD_{A_k}$ into $\asD'$ given by
$H_k=\Pi_{\asD'_k}^*\ga_k$ and $H\in\eUM^2_\tu{ball}(\asD,\asD')$ by
\[
H=\mat{ccccc}{\cdots&H_{-1}&H_0&H_1&\cdots}.
\]
\end{proc}

One can reverse this procedure in order to obtain a set of
contractions $\{B_n\mid n\in\BZ\}$  that satisfy (\ref{reqB}) for
each $k\in\BZ$ from a operator matrix in
$\eUM^2_\tu{ball}(\asD,\asD')$.

\begin{proc}\label{pc:HtoB}
Let $H\in\eUM^2_\tu{ball}(\asD,\asD')$. For each $k\in\BZ$ let $H_k$
be the $k$-th column of $H$, i.e., $H_k=H|_{\sD_{A_k}}$, set
$\ga_k=\Pi_{\asD_k'}H_k$ and define $B_k$ by \tu{(\ref{GaB})}.
\end{proc}

It is straightforward that $\ga_k$ defined in Procedure
\ref{pc:HtoB} is a contraction, and thus that $B_k$ is a contraction
satisfying the first requirement of (\ref{reqB}), and moreover,
these procedures are each others inverse. Furthermore, the condition
(\ref{sol1}) on the contractions $\{\ga_n\mid n\in\BZ\}$ translates
to the columns in the operator matrix $H$ obtained by Procedure
\ref{pc:BtoH} in the form
\begin{equation}\label{sol2}
H_k|_{\sF_k}=\tau_{k}'\o_{k,1}+H_{k-1}\o_{k,2},
\end{equation}
where $\tau_{k}'$ is the canonical embedding of $\sD_{T'_n}$ into
$\asD'$, or, equivalently, in the form (\ref{intpolcond1}) with
$\o_{k,1}$ and $\o_{k,2}$ as defined in the present subsection,
$\sU_k=\sD_{A_k}$ and $\sY_k=\sD_{T_k'}$ for each $k\in\BZ$. The
converse is also true. For an operator matrix
$H\in\eUM^2_\tu{ball}(\asD,\asD')$ such that the $k$-th column $H_k$
of $H$ satisfies (\ref{sol2}), the contraction $\ga_k$ obtained by
Procedure \ref{pc:HtoB} satisfies (\ref{sol1}). In conclusion, we
have the following result.

\begin{thm}\label{T:TV-RCL-sols}
Let $\la=\{A_n,T'_n,U'_n,R_n,Q_n\mid n\in\BZ\}$ be a data set as
described above. For any solution $H$ of Problem \ref{mainprobl}
with $\o_k$ equal to the $k$-th underlying contraction of $\la$ for
each $k\in\BZ$, the set $\{B_n\mid n\in\BZ\}$ obtained from
Procedure \ref{pc:HtoB} is a solution of Problem
\ref{TV-RCLProblem}. Conversely, for any solution $\{B_n\mid
n\in\BZ\}$ of Problem \ref{TV-RCLProblem}, the operator matrix $H$
obtained from Procedure \ref{pc:BtoH} is a solution of Problem
\ref{mainprobl} in case $\o_k$ in \tu{(\ref{defomk})} is the $k$-th
underlying contraction of $\la$ for each $k\in\BZ$.
 \end{thm}

{}From Theorem \ref{T:TV-RCL-sols} in combination with Theorems
\ref{thmmain1} and \ref{thmmain2} it is clear how all solutions of
Problem \ref{TV-RCLProblem} can be described. Furthermore, in
combination with the full version of the second main result, Theorem
\ref{thmmain2full}, a time-variant analogue of Theorem 1.2 in
\cite{FtHK06b} (see also Theorem 1.2 in  \cite{tH07}) is obtained.

Not only can Problem \ref{TV-RCLProblem} be seen as a special case
of Problem \ref{mainprobl}, the converse is also true, as explained
in the following theorem.

\begin{thm}\label{T:TV-RCL-converse}
For each $k\in\BZ$ let $\o_k$ be a contraction of the form
\tu{(\ref{defomk})} with $\sF_k$ a subspace of $\sU_k$. Set
\[
\begin{array}{c}
A_k=\mat{cc}{I_{\sY_{k+1}}&0}:\mat{c}{\sY_{k+1}\\\sU_k}\to\sY_{k+1},\quad
T_k'=0:\sY_k\to\sY_{k+1},\\[.4cm]
R_{k}=\mat{c}{\o_{1,k+1}\\\o_{2,k+1}}:\sF_{k+1}\to\mat{c}{\sY_{k+1}\\\sY_{k}},\quad
Q_{k}=\mat{c}{0\\\Pi_{\sF_k}^*}:\sF_k\to\mat{c}{\sY_{k+1}\\\sY_{k}},
\end{array}
\]
and define $U_k'$ by \tu{(\ref{U'k})} for each $k\in\BZ$. Here
$\Pi_{\sF_k}$ denotes the orthogonal projection from $\sU_k$ onto
$\sF_k$. Then $\la=\{A_n,T'_n,U'_n,R_n,Q_n\mid n\in\BZ\}$ is a data
set for a time-variant relaxed commutant lifting problem, and $\o_k$
is the $k$-th underlying contraction of $\la$ for each $k\in\BZ$.
\end{thm}

\begin{proof}
Fix $k\in\BZ$. Clearly, $T'_k$ and $A_k$ are contraction. Moreover,
$T'_kA_{k-1}R_{k-1}$ and $A_{n}Q_{n}$ are both equal to the zero
operator from $\sF_k$ into $\sY_{k+1}$, and, since $R_{k-1}=\o_k$ is
a contraction, we have $R_{k-1}^*R_{k-1}\leq I_{\sF_k}=Q_k^*Q_k$. It
then follows that $\la$ is a data set for a time-variant relaxed
commutant lifting problem. Next, observe that
\[
D_{A_k}Q_k=\Pi_{\sF_k},\quad D_{T'_k}A_{k-1}R_{k-1}=\o_{1,k}\ands
D_{A_{k-1}}R_{k-1}=\o_{2,k}.
\]
This implies that the $k$-th underlying contraction of $\la$ is
equal to $\o_k$.
\end{proof}

\end{document}